\pdfoutput=1
\RequirePackage{ifpdf}
\ifpdf 
\documentclass[pdftex]{sigma}
\else
\documentclass{sigma}
\fi

\usepackage{mathrsfs,mathtools}

\usepackage[all]{xy}

\numberwithin{equation}{section}

\newtheorem{Theorem}{Theorem}[section]
\newtheorem*{Theorem*}{Theorem}
\newtheorem{Corollary}[Theorem]{Corollary}
\newtheorem{Lemma}[Theorem]{Lemma}
\newtheorem{Proposition}[Theorem]{Proposition}
 { \theoremstyle{definition}
\newtheorem{Definition}[Theorem]{Definition}

\newtheorem{Remark}[Theorem]{Remark}

\newtheorem{Question}[Theorem]{Question}
\newtheorem{Assumption}[Theorem]{Assumption}}

\newcommand{\C}{\mathbb{C}}
\newcommand{\Z}{\mathbb{Z}}
\newcommand{\R}{\mathbb{R}}

\newcommand{\ps}{\mathbb{P}}
\newcommand{\id}{\mathrm{id}}

\newcommand{\Image}{\operatorname{Im}}
\newcommand{\Ker}{\operatorname{Ker}}
\newcommand{\rk}{\operatorname{rank}}
\newcommand{\Hom}{\operatorname{Hom}}
\newcommand{\End}{\operatorname{End}}

\newcommand{\ord}{\operatorname{ord}}
\newcommand{\Tr}{\operatorname{Tr}}

\newcommand{\sheafhom}{\operatorname{\mathcal{H}{\rm om}}}
\newcommand{\strsheaf}{\mathcal{O}}

\newcommand{\gr}{\operatorname{\mathrm{Gr}}}
\newcommand{\pa}{\operatorname{\mathcal{P}ar}}

\begin{document}

\allowdisplaybreaks

\newcommand{\arXivNumber}{2403.07258}

\renewcommand{\PaperNumber}{111}

\FirstPageHeading

\ShortArticleName{Harmonic Metrics for Higgs Bundles of Rank 3 in the Hitchin Section}

\ArticleName{Harmonic Metrics for Higgs Bundles of Rank 3\\in the Hitchin Section}

\Author{Hitoshi FUJIOKA}

\AuthorNameForHeading{H.~Fujioka}

\Address{Research Institute for Mathematical Sciences, Kyoto University, Kyoto 606-8502, Japan}
\Email{\href{mailto:fujioka@kurims.kyoto-u.ac.jp}{fujioka@kurims.kyoto-u.ac.jp}}

\ArticleDates{Received April 21, 2024, in final form November 29, 2024; Published online December 11, 2024}

\Abstract{Given a tuple of holomorphic differentials on a Riemann surface, one can define a Higgs bundle in the Hitchin section and a natural symmetric pairing of the Higgs bundle. We study whether a Higgs bundle of rank 3 in the Hitchin section has a compatible harmonic metric when the spectral curve is a 2-sheeted branched covering of the Riemann surface. In~particular, we give a condition for Higgs bundles in the Hitchin section on $\mathbb{C}$ or $\mathbb{C}^*$ to have compatible harmonic metrics.}


\Classification{53C07; 14D21}

\section{Introduction}
\subsection{Harmonic bundles}
Suppose that $X$ is a Riemann surface. Let $E$ be a holomorphic vector bundle on $X$ and ${\theta\in H^0(X,\End{E}\otimes K_X)}$, where $K_X$ is the canonical line bundle of $X$. The pair $(E,\theta)$ is called a~Higgs bundle and $\theta$ is called a~Higgs field.
\begin{Definition}
 Let $(E,\theta)$ be a Higgs bundle over $X$. A Hermitian metric $h$ of $(E,\theta)$ is called harmonic if
 \begin{equation}\label{Hitchin_eq}
 F_{\nabla_h} + \bigl[\theta,\theta^{\ast h}\bigr] = 0,
 \end{equation}
 where $\nabla_h$ is the Chern connection of $(E,h)$, $F_{\nabla_h}$ is the curvature of $\nabla_h$ and $\theta^{\ast h}$ is the adjoint of $\theta$ with respect to $h$. Such a tuple $(E,\theta,h)$ is called a harmonic bundle.
\end{Definition}
\begin{Remark}
 When $X$ is compact, for a $\C$-vector bundle $E$ on $X$, the degree $\deg{E}$ is defined by \smash{$\int_X c_1(E)$}.
 The equation (\ref{Hitchin_eq}) means that the connection $\nabla_h + \theta + \theta^{*h}$ is flat. Therefore, we implicitly focus on holomorphic bundles of degree 0.
\end{Remark}
The trivial bundle $\strsheaf_X = X\times \C$ has the trivial Hermitian metric $h_X$ defined by $h_X(a,b) = a\Bar{b}$. If a harmonic bundle $(E,\theta,h)$ of rank $r$ on $X$ satisfies $\det(E) = \strsheaf_X$, $\Tr{\theta}=0$ and $\det(h) = h_X$, $(E,\theta,h)$ is called an ${\rm SL}(r,\C)$-harmonic bundle.

The equation (\ref{Hitchin_eq}) was introduced by Hitchin in \cite{Hi1} as the dimensional reduction of the self-duality equations. It is an important question whether a harmonic metric exists for a given Higgs bundle. Because (\ref{Hitchin_eq}) is a nonlinear partial differential equation, it is difficult to solve it for general Higgs bundles. For Higgs bundles on compact Riemann surfaces, the following theorem, due to Hitchin and Simpson, establishes the existence and the uniqueness of harmonic metrics, and it is the most fundamental theorem about harmonic bundles.
\begin{Theorem}[Hitchin \cite{Hi1} and Simpson \cite{Si1}]\label{Hitchin_thm1}
 Suppose that $X$ is compact. A Higgs bundle $(E,\theta)$ on~$X$ has a harmonic metric if and only if $(E,\theta)$ is polystable and $\deg{E}=0$. Moreover, if $h_1$ and~$h_2$ are harmonic metrics of $(E,\theta)$, then there exists a decomposition
 \begin{equation*}
 (E,\theta) = \bigoplus_{i=1}^{n} (E_i,\theta|_{E_i}),
 \end{equation*}
 such that $c_i h_1|_{E_i} = h_2|_{E_i}$ for positive constants $c_1,c_2,\ldots,c_n$ and this decomposition is orthogonal with respect to $h_1$ and $h_2$.
\end{Theorem}

For general Higgs bundles on noncompact Riemann surfaces, it is unknown when a harmonic metric exists.
Suppose that $X$ is the complement of a finite subset in a compact Riemann surface. In \cite{Si1}, Simpson gave a sufficient condition for the existence of a harmonic metric of a~Higgs bundle $(E,\theta)$ on $X$. Let $g_X$ be a K\"{a}hler metric satisfying the condition in \cite[Proposition~2.4]{Si1}. Let $h_0$ be a Hermitian metric of $E$ such that $F_{\nabla_{h}} + \bigl[\theta, \theta^{\ast h}\bigr]$ is bounded with respect to~$h_0$ and~$g_X$ . We~can define the stability condition for $(E,\theta, h_0)$ as explained in \cite[Section 3]{Si1}. Suppose~$\det(h_0)$ is flat for simplicity. If $(E,\theta, h_0)$ is stable, then there exists a harmonic metric $h$ of $(E,\theta)$ with~$h_0$ and~$h$ is mutually bounded \cite[Theorem 1]{Si1}. In addition, Theorem \ref{Hitchin_thm1} is generalized also for filtered Higgs bundles. It was proved by Simpson~\cite{Si2} in the tame case and by Biquard--Boalch \cite{BB} and Mochizuki \cite{Mo1} in the wild case (see Section~\ref{sec2} for details).

\subsection{Higgs bundles in the Hitchin section}
\subsubsection{Non-degenerate symmetric pairings}
For a Higgs bundle $(E,\theta)$, a holomorphic symmetric pairing $C$ of $E$ is called a symmetric pairing of $(E,\theta)$ if $C(\id\otimes\theta) = C(\theta\otimes \id)$. A harmonic metric $h$ of $(E,\theta)$ is said to be compatible with a~non-degenerate symmetric pairing $C$ of $(E,\theta)$ if the isomorphism $E\to E^{*}$, $v\mapsto (w\mapsto C(w,v))$ is an isometry with respect to $h$ and $h^{*}$, where $h^{*}$ is the Hermitian metric on $E^{*}$ induced by $h$.

\subsubsection{Higgs bundles in the Hitchin section}
Let $q_j$ be a holomorphic $j$-differential on $X$ for $j=2,\ldots,r$. The tuple $(q_2,\dots,q_r)$ is denoted by~$q$. We fix a line bundle \smash{$K_X^{1/2}$}. The holomorphic differential $q_j$ induces the morphisms
\begin{equation*}
 K_X ^{(r-2i+1)/2} \longrightarrow K_X ^{(r-2i+2j-1)/2}\otimes K_X
\end{equation*}
for $i=j,\ldots,r$. Then we can construct a Higgs bundle $(\mathbb{K}_{X,r},\theta(q))$ of rank $r$ as follows:
\begin{equation*}
 \mathbb{K}_{X,r} = \bigoplus_{i=1}^r K_X^{(r-2i+1)/2},\qquad
 \theta(q) =
 \begin{pmatrix}
 0 & q_2 & q_3 &\cdots & q_r \\
 1 & 0 & q_2 & \cdots & q_{r-1} \\
 0 & 1 & \ddots & \ddots & \vdots \\
 \vdots & \ddots & \ddots & \ddots & q_2 \\
 0 & \cdots & 0 & 1 & 0
\end{pmatrix}.
\end{equation*}
These Higgs bundles $(\mathbb{K}_{X,r},\theta(q))$ were introduced by Hitchin in \cite{Hi2} and are called Higgs bundles in the Hitchin section. If $X$ is compact, then a mapping, which is called the Hitchin fibration, from the moduli space $\mathcal{M}$ of polystable ${\rm SL}(r,\mathbb{C})$-Higgs bundles on $X$ to $\bigoplus_{i=2}^r H^0\bigl(X,K_X^{\otimes i}\bigr)$ is constructed by assigning to $[(E,\theta)] \in \mathcal{M}$ the coefficients of the characteristic polynomial of $\theta$, and these Higgs bundles $(\mathbb{K}_{X,r},\theta(q))$ form a right-inverse of the Hitchin fibration.

\subsubsection{Symmetric pairings of Higgs bundles in the Hitchin section}
A Higgs bundle $(\mathbb{K}_{X,r},\theta(q))$ in the Hitchin section has a natural symmetric pairing $C_{X,r}$ defined by the~morphism
\[
K_X^{(r-2i+1)/2} \otimes K_X^{-(r-2i+1)/2} \to \mathcal{O}_X.
\]
 Suppose that $z$ is a local holomorphic~coordinate and \smash{$({\rm d}z)^{1/2}$} is a local frame of \smash{$K_X^{1/2}$}, the pairing $C_{X,2}$ on \smash{$K_X^{1/2}\oplus K_X^{-1/2}$} is given~by
\begin{equation*}
 \begin{pmatrix}
 0 & 1 \\
 1 & 0
\end{pmatrix}
\end{equation*}
with respect to the frame $\bigl(({\rm d}z)^{1/2},({\rm d}z)^{-1/2}\bigr)$. The pairing $C_{X,3}$ on $K_X\oplus \strsheaf_X \oplus K_X^{-1}$ is given~by
\begin{equation*}
\begin{pmatrix}
 0 & 0 & 1 \\
 0 & 1 & 0 \\
 1 & 0 & 0
\end{pmatrix}
\end{equation*}
with respect to the frame $\bigl({\rm d}z,1,({\rm d}z)^{-1}\bigr)$.

\subsection{Main results}
We will consider the following question.
\begin{Question}\label{question}
 Suppose that $X$ is a noncompact parabolic Riemann surface, i.e., $X=\C$ or~$\C^*$.
 Given a tuple of holomorphic polynomial differentials $q=(q_2,\dots,q_r)$ on $X$, i.e., meromorphic differentials on $\ps^1$ which have a possible pole at $\infty$, does there exist a harmonic metric of~$(\mathbb{K}_{X,r},\theta(q))$ compatible with $C_{X,r}$?
\end{Question}
\begin{Remark}
In \cite{LM2}, Li and Mochizuki proved that if $X$ is hyperbolic, then there exists a~harmonic metric of $(\mathbb{K}_{X,r},\theta(q))$ compatible with $C_{X,r}$.
\end{Remark}
In general, for a Higgs bundle $(E,\theta)$ on $X$, the spectral curve $\Sigma_{E,\theta}$ of $(E,\theta)$ is a curve defined by $\{ t\in K_{X} \mid \det(t\,\id_E -\theta)=0\}$. We can define the natural projection from $\Sigma_{E,\theta}$ to $X$ by restricting $K_X\to X$ to $\Sigma_{E,\theta}$. The spectral curve of $(\mathbb{K}_{X,r},\theta(q))$ is denoted by $\Sigma_{X,q}$. We call the natural projection $\pi\colon \Sigma_{X,q}\to X$ an $n$-sheeted branched covering if the fiber of $\pi\colon \Sigma_{X,q}\to X$ with the largest cardinality consists of $n$ elements. Let $z$ be the holomorphic coordinate of $\C$. The following theorems are our main results, giving a partial answer to Question \ref{question}.
\begin{Theorem}\label{main}
 If the natural projection $\pi\colon \Sigma_{\C,q}\to \C$ is a one- or two-sheeted branched covering, then there exists a polynomial $f\in \C[z]$ such that
 \begin{equation*}
 q_2 = 3\cdot 2^{-5/3}f^2 ({\rm d}z)^2, \qquad q_3 = f^3 ({\rm d}z)^3.
 \end{equation*}
 If $\deg{f}\geq 2$, then there exists a harmonic metric of $(\mathbb{K}_{\C,3},\theta(q))$ compatible with $C_{\C,3}$. Otherwise, there does not exist any harmonic metric compatible with $C_{\C,3}$.
\end{Theorem}
\begin{Theorem}\label{main_Cstar}
 If the natural projection $\pi\colon \Sigma_{\C^*,q}\to \C^*$ is a one- or two-sheeted branched covering, then there exists $f\in \C\bigl[z,z^{-1}\bigr]$ such that
 \begin{equation*}
 q_2 = 3\cdot 2^{-5/3}f^2 ({\rm d}z/z)^2, \qquad q_3 = f^3 ({\rm d}z/z)^3.
 \end{equation*}
 If $f$ is not constant, then there exists a harmonic metric of $(\mathbb{K}_{\C^*,3},\theta(q))$ compatible with $C_{\C^*,3}$. If~$f$~is~constant, then there does not exist any harmonic metric compatible with $C_{\C^*,3}$.
\end{Theorem}

{\samepage\begin{Remark}\quad
 \begin{itemize}\itemsep=0pt
 \item In \cite{LM1}, Li and Mochizuki proved that a Higgs bundle $(E,\theta)$ with a symmetric pairing admits a compatible harmonic metric if $(E,\theta)$ is generically regular semisimple using an analytical approach. In particular, if $\pi$ is a 3-sheeted branched covering, then $(\mathbb{K}_{X,3},\theta(q))$ has a harmonic metric compatible with $C_{X,3}$. Therefore, we obtain the complete answer to the $r=3$ case of Question \ref{question}.
 \item In the $r=2$ case, $(\mathbb{K}_{X,2},\theta(q))$ admits a compatible harmonic metric if $\pi$ is a two-sheeted branched covering by Li--Mochizuki \cite{LM1}. Otherwise, it has no such metric by \cite[Propositions~3.41 and~3.42]{LM3}.
 \item As we will see in Remark \ref{rem_classifying}, compatible harmonic metrics of $(\mathbb{K}_{\C,3},\theta(q))$ correspond to filtered Higgs bundles satisfying some condition, which is parameterized by a real number in $(2,(\deg{f}+3)/2]$. Therefore, the uniqueness of compatible harmonic metrics is not true.
 \item We need compatibility to prove the non-existence of harmonic metrics in Theorems~\ref{main} and~\ref{main_Cstar}. The author does not know whether there exists a harmonic metric if $(\mathbb{K}_{X,3},\theta(q))$ admits no compatible harmonic metric.
 \end{itemize}
\end{Remark}}

To prove Theorem \ref{main}, we use the following theorem proved by Li and Mochizuki. See Section~\ref{sec2} for detailed definitions.
\begin{Theorem}[\cite{LM1}]\label{correspond}
 Suppose that $X$ is compact. The following two objects are equivalent:
 \begin{itemize}\itemsep=0pt
 \item Wild harmonic bundles on $(X,D)$ compatible with a non-degenerate symmetric pairing.
 \item Good polystable filtered Higgs bundles on $(X,D)$ of degree 0 with a perfect symmetric pairing.
 \end{itemize}
\end{Theorem}
By Theorem \ref{correspond}, it suffices to study when there exists a stable good filtered extension of $(\mathbb{K}_{X,3},\theta(q))$ such that $C_{X,3}$ extends to a perfect pairing. The proof of Theorem \ref{main} is outlined as follows.
\begin{itemize}\itemsep=0pt\setlength{\leftskip}{0.46cm}
 \item[{\bf Step 1}] We calculate the eigenvalues of the Higgs filed $\theta(q)$, and take sections $s_1$, $s_2$ and $s_3$ of~$\mathbb{K}_{\C,3}$, which form a frame around $\infty$ such that the representation of $\theta(q)$ is
 a Jordan canonical form with respect to $(s_1,s_2,s_3)$.
 \item[{\bf Step 2}] Using $(s_1,s_2,s_3)$, we construct a meromorphic extension of $\mathbb{K}_{\C,3}$ and a filtered bundle~$\mathcal{P}^{d}_{*}\mathbb{K}_{\C,3}$ over $\mathbb{K}_{\C,3}$ depending $d=(d_1,d_2,d_3)\in \R^3$ such that
 \begin{equation*}
 \mathcal{P}^{d}_{c}(\mathbb{K}_{\C,3})_{\infty} = \bigoplus_{i=1}^3\strsheaf_{\ps^1,\infty}([c-d_i]\infty)\cdot s_i,
 \end{equation*}
 where for $c\in\R$, $[c]$ is an integer satisfying $c-1<[c]\leq c$.
 \item[{\bf Step 3}] We find the condition for $d$ that $\bigl(\mathcal{P}^{d}_{*}\mathbb{K}_{\C,3},\theta(q)\bigr)$ is good and stable, and that the pairing~$C_{\C,3}$ extends to a perfect pairing. We can prove that if $\deg{f}\geq 2$, then there exists~$d\in\R^3$ satisfying the condition.
 \item[{\bf Step 4}] We prove that if a filtered Higgs bundle is a good filtered extension of $(\mathbb{K}_{\C,3},\theta(q))$ and the pairing induced by $C_{\C,3}$ is perfect, then there exists $d\in\R^3$ such that the filtered bundle is $\mathcal{P}^{d}_{*}(\mathbb{K}_{\C,3})$. We can prove that if $\deg{f}\leq 1$, then there is no harmonic metric of $(\mathbb{K}_{\C,3},\theta(q))$ compatible with $C$ from this.
\end{itemize}
In addition to Theorem \ref{main}, we obtain the classification of filtered Higgs bundles corresponding to compatible harmonic metrics of $(\mathbb{K}_{\C,3},\theta(q))$ in the case when $\pi\colon \Sigma_{\C,q}\to \C$ is a two-sheeted branched covering.
Moreover, the proof of Theorem \ref{main} can be partially generalized to the case of the Higgs bundle $(\mathbb{K}_{X,3},\theta(q))$ whose Higgs filed is meromorphic, where $X$ is any other quasi-projective curve.

The proof of Theorem \ref{main_Cstar} is outlined as follows.
\begin{itemize}\itemsep=0pt\setlength{\leftskip}{0.55cm}
 \item[{\bf Step 1$'$}] We see that there exists a harmonic metric compatible with $C_{\C^*,3}$ unless $f(z)= az^b$, $a\in \C^*$, $|b|\geq 3$, in the same way as Step 1, Step 2 and Step 3 in the outline of the proof of Theorem \ref{main}.
 \item[{\bf Step 2$'$}] We see that there does not exist any harmonic metric compatible with $C_{\C^*,3}$ if $f$ is constant in a manner analogous to Step 4 in the outline of the proof of Theorem \ref{main}.
 \item[{\bf Step 3$'$}] In the case of $f(z)=az^b$, $a\in \C^*$, $|b|\geq 3$, we explicitly construct a stable good filtered extension of $(\mathbb{K}_{X,3},\theta(q))$ such that the pairing induced by $C_{\C^*,3}$ become perfect.
\end{itemize}

In the general rank case, there are many cases about the multiplicities of the eigenvalues of the Higgs field. For example, the $r=4$ case is divided into the case when the two eigenvalues with the multiplicity 2 appear, the case when the one eigenvalue with the multiplicity 1 appears and the ones of the others are 1, etc. Because of this difficulty, the author has not yet obtained similar results for the general rank.

\section{Preliminaries}\label{sec2}
\subsection{Filtered Higgs bundles}
\subsubsection{Filtered bundles}
We will review the notion of filtered bundles, following \cite{Mo1,Si1,Si2}. In this paper, the notations are based on \cite{LM1}.
Let $X$ be a Riemann surface, and let $D$ be a discrete subset of $X$. We write $\mathcal{O}_{X}(*D)$ for the sheaf of meromorphic functions on $X$ whose poles are contained in $D$. For any sheaf $\mathcal{F}$, we write $\mathcal{F}_p$ for the stalk of $\mathcal{F}$ at $p\in X$.
\begin{Definition}
 A filtered bundle $\mathcal{P}_* E$ on $(X,D)$ is a locally free $\mathcal{O}_{X}(*D)$-module $E$ with a~tuple of filtrations $\{\mathcal{P}_{c} E_p\}_{c\in \R}$ of free $\strsheaf_{X,p}$-submodules of $E_p$ for $p\in D$ such that for $c$, $c_1$ and~$c_2\in \R$,
 \begin{enumerate}\itemsep=0pt
 \item[1)] $(\mathcal{P}_{c} E_p)(*p) = E_p$,
 \item[2)] $\mathcal{P}_{c_1} E_p \subset \mathcal{P}_{c_2} E_p$ if $c_1 \leq c_2$,
 \item[3)] $\mathcal{P}_{c} E_p =\bigcap_{a>c} (\mathcal{P}_{a} E_p)$,
 \item[4)] $\mathcal{P}_{c+1} E_p = \mathcal{P}_{c} E_p \otimes_{\mathcal{O}_{X,p}} \mathcal{O}_{X,p}(p)$.
 \end{enumerate}
 The rank of $\mathcal{P}_* E$ is defined to be the rank of $E$. We also say that $\mathcal{P}_* E$ is a filtered bundle over~$E$.
 We define $\mathcal{P}_{<c} E_p$ to be $\bigcup_{a<c} (\mathcal{P}_a E_p)$ and $\gr^{\mathcal{P}}_{c} (E_p)$ to be $\mathcal{P}_{c} E_p /\mathcal{P}_{<c} E_p$. Let $\pa(\mathcal{P}_{*} E_p)$ denote \smash{$\bigl\{c\in\R \mid \gr^{\mathcal{P}}_{c} (E_p)\neq 0\bigr\}$}.
\end{Definition}
For $c=\bigl(c^{(p)}\bigr)_{p\in D} \in \R^{D}$, let $\mathcal{P}_{c} E$ denote the locally free $\strsheaf_{X}$-submodule of $E$ such that if $p\in D$, $(\mathcal{P}_{c} E)_{p} = \mathcal{P}_{c^{(p)}} E_p$, otherwise $(\mathcal{P}_{c} E)_{p} = E_p$.
\begin{Definition}
 Let $\mathcal{P}_* E_1$ and $\mathcal{P}_* E_2$ be filtered bundles on $(X,D)$. A morphism $\varphi$ from $\mathcal{P}_* E_1$ to $\mathcal{P}_* E_2$ is a morphism of sheaves of $\mathcal{O}_{X}(*D)$-modules $\varphi\colon E_1\to E_2$ such that $\varphi(\mathcal{P}_{c} E_{1}) \subset \mathcal{P}_{c} E_{2}$ for any $c\in \R^{D}$.
\end{Definition}
Let $\mathcal{P}_{*}^{(0)} (\strsheaf_{X} (*D))$ denote $\strsheaf_{X}(*D)$ with filtrations defined by
\[
\mathcal{P}_{c}^{(0)} (\strsheaf_{X}(*D)) = \strsheaf_{X} \biggl(\sum_{p\in D} \bigl[c^{(p)}\bigr]p\biggr).
\]
Let $\mathcal{P}_* E_1$ and $\mathcal{P}_* E_2$ be filtered bundles. We can define some filtered bundles induced by $\mathcal{P}_* E_1$ and $\mathcal{P}_* E_2$.
\begin{itemize}\itemsep=0pt
 \item Direct sum
$\mathcal{P}_{c}(E_1\oplus E_2)_p = \mathcal{P}_{c}E_{1,p} \oplus \mathcal{P}_{c}E_{2,p}$.
 \item Tensor product
$\mathcal{P}_{c}(E_1\otimes E_2)_p = \sum_{a+b\leq c} (\mathcal{P}_{a}E_{1,p} \otimes \mathcal{P}_{b}E_{2,p})$.
 \item Hom
$\mathcal{P}_{c}(\sheafhom(E_1,E_2))_p = \{f\in \Hom_{\strsheaf_{X,p}}(E_{1,p},E_{2,p})\mid f(\mathcal{P}_{a}E_{1})\subset\mathcal{P}_{a+c}E_{2} \}$.
\end{itemize}
We write $\mathcal{P}_{*}E_1\oplus \mathcal{P}_{*}E_2$, $\mathcal{P}_{*}E_1\otimes \mathcal{P}_{*}E_2$ and $\sheafhom(\mathcal{P}_{*} E_1,\mathcal{P}_{*} E_2)$ for $\mathcal{P}_{*}(E_1\oplus E_2)$, $\mathcal{P}_{*}(E_1\otimes E_2)$, and $\mathcal{P}_{*}(\sheafhom(E_1,E_2))$. We define the dual bundle of $\mathcal{P}_{*} E_{1}$ as $\sheafhom\bigl(\mathcal{P}_{*} E_1,\mathcal{P}_{*}^{(0)} (\strsheaf_{X} (*D))\bigr)$.

\subsubsection{Filtered Higgs bundles}
\begin{Definition}
 Suppose that $X$ is compact.
 The degree $\deg{\mathcal{P}_{*}E}$ of a filtered bundle $\mathcal{P}_{*} E$ on~$(X,D)$ is
 \[
 \deg{\mathcal{P}_{*}E} = \deg{\mathcal{P}_{c}E} - \sum_{p\in D}\sum_{c^{(p)}-1<a\leq c^{(p)}} a\dim\gr_{a}^{\mathcal{P}} (E_p)
 \]
for $c\in\R^D$. We write $\mu(\mathcal{P}_{*} E)$ for $\deg{\mathcal{P}_{*} E}/\rk{E}$, which is called the slope of $\mathcal{P}_{*} E$.
\end{Definition}

Let $\mathcal{P}_{*} E$ be a filtered bundle on $(X,D)$. Let $E'$ be a locally free $\strsheaf_X(*D)$-submodule of $E$ such that $E/E'$ is also locally free. Then the filtered bundles $\mathcal{P}_{*} E'$ and $\mathcal{P}_{*} (E/E')$ are induced as follows: for $c\in \R^D$,
\begin{equation*}
 \mathcal{P}_{c}E' = (\mathcal{P}_{c}E) \cap E',\qquad \mathcal{P}_{c}(E/E') = \Image\bigl(\mathcal{P}_{c}E \to E/E'\bigr).
\end{equation*}
\begin{Lemma}\label{degree_proof}
 Suppose that $X$ is compact. Then the following holds:
 \begin{equation*}
 \deg{\mathcal{P}_{*} E} = \deg{\mathcal{P}_{*} E'} + \deg{\mathcal{P}_{*} \bigl(E/E'\bigr)}.
 \end{equation*}
\end{Lemma}
\begin{proof}
 Though this lemma is well known, we provide proof for the convenience of the reader.
 By the short exact sequence
$ 0\longrightarrow \mathcal{P}_{c}E' \longrightarrow \mathcal{P}_{c}E \longrightarrow \mathcal{P}_{c}\bigl(E/E'\bigr) \longrightarrow 0$,
 we obtain
 \begin{equation*}
 \deg{\mathcal{P}_{c}E} = \deg{\mathcal{P}_{c}E'} + \deg{\mathcal{P}_{c}\bigl(E/E'\bigr)}.
 \end{equation*}
 In addition, for $a\in\R$, we obtain the following commutative diagram:
 \begin{equation*}
 \begin{CD}
 @. 0 @. 0 @. 0 \\
 @. @VVV @VVV @VVV \\
 0 @>>> \mathcal{P}_{<a}E'_p @>>> \mathcal{P}_{<a}{E}_p @>>> \mathcal{P}_{<a}(E/E')_p @>>> 0 \\
 @. @VVV @VVV @VVV \\
 0 @>>> \mathcal{P}_{a}E'_p @>>> \mathcal{P}_{a}E_p @>>> \mathcal{P}_{a}(E/E')_p @>>> 0 \\
 @. @VVV @VVV @VVV \\
 0 @>>> \gr^{\mathcal{P}}_{a}\bigl(E'_p\bigr) @>>> \gr^{\mathcal{P}}_{a}(E_p) @>>> \gr^{\mathcal{P}}_{a}(E/E')_p @>>> 0. \\
 @. @VVV @VVV @VVV \\
 @. 0 @. 0 @. 0 \\
 \end{CD}
 \end{equation*}
 Since all columns and the top two rows are exact, the bottom row is exact. Therefore, we obtain $\dim{\gr^{\mathcal{P}}_{a}(E_p)} = \dim{\gr^{\mathcal{P}}_{a}\bigl(E'_p\bigr)} + \dim{\gr^{\mathcal{P}}_{a}(E/E')_p}$.
 As a result, we obtain $\deg{\mathcal{P}_{*} E} = \deg{\mathcal{P}_{*} E'} + \deg{\mathcal{P}_{*} (E/E')}$.
\end{proof}

We review the notion of filtered Higgs bundles and define the notion of stability of filtered Higgs bundles.
\begin{Definition}
 A filtered Higgs bundle on $(X,D)$ is a pair of a filtered bundle $\mathcal{P}_{*} E$ on $(X,D)$ and a morphism $\theta\colon E\to E\otimes_{\strsheaf_{X}} K_{X}$. The morphism $\theta$ is called a Higgs field.
\end{Definition}

\begin{Definition}
 Suppose that $X$ is compact.
 A filtered Higgs bundle $(\mathcal{P}_{*} E,\theta)$ on $(X,D)$ is stable (resp. semistable) if for any
 proper filtered Higgs subbundles $(\mathcal{P}_{*} E',\theta')$ of positive rank of $(\mathcal{P}_{*} E,\theta)$,
 \begin{equation*}
 \mu(E) > \mu\bigl(E'\bigr)\qquad \bigl(\text{resp.}\ \mu(E) \geq \mu\bigl(E'\bigr)\bigr).
 \end{equation*}
 A filtered Higgs bundle $(\mathcal{P}_{*} E,\theta)$ is polystable if $(\mathcal{P}_{*} E,\theta)$ is a direct sum of stable filtered Higgs bundles with the same slopes.
\end{Definition}
\begin{Definition}
 A filtered Higgs bundle $(\mathcal{P}_{*} E,\theta)$ on $(X,D)$ is called regular if $\theta$ is logarithmic with respect to the filtrations, that is, $\theta(\mathcal{P}_{c} E_p)\subset \mathcal{P}_{c+1} E_p\otimes K_{X,p}$ for any $p\in D$ and any $c\in \R$.
\end{Definition}
Let $\varphi_m(w)=w^m$ be a ramified covering. We define a filtration of $\varphi_m ^{*}E_p$ to be
\begin{equation*}
 \mathcal{P}_{c}\bigl(\varphi_m ^{*} E_p\bigr) = \sum_{ma+n\leq c} \varphi_m ^{*}\bigl(\mathcal{P}_{a} E_p\bigr)\otimes \strsheaf_{X,p}(np)
\end{equation*}
for $c\in\R$.
\begin{Definition}
 A filtered Higgs bundle $(\mathcal{P}_{*} E,\theta)$ on $(X,D)$ is called good if for any $p\in D$ there exists a ramified covering $\varphi_{p,m}(z_{p,m})=z_{p,m}^m$ such that $\bigl(\mathcal{P}_{*}\bigl(\varphi_{p,m}^{*}E_p\bigr),\varphi_{p,m}^{*}\theta\bigr)$ has a decomposition
 \begin{equation*}
 \bigl(\mathcal{P}_{*}\bigl(\varphi_{p,m}^{*}E_p\bigr),\varphi_{p,m}^{*}\theta\bigr) = \bigoplus_{\mathfrak{a} \in z_{p,m}^{-1}\mathbb{C}[z_{p,m}^{-1}]} (\mathcal{P}_{*} E_{\mathfrak{a},p},\theta_{\mathfrak{a}}),
 \end{equation*}
 which satisfies that $\theta_{\mathfrak{a}} - \mathrm{d}\mathfrak{a}\, \id$ is logarithmic at $p$ with respect to the lattices $\mathcal{P}_{b}E_{\mathfrak{a},p}$.
\end{Definition}

\subsection{Non-degenerate symmetric pairings}
We introduce the notion of symmetric pairings of Higgs bundles in this section by following \cite{LM1} and theorems used in a later section.
\subsubsection{Symmetric pairings on Higgs bundles}
\begin{Definition}
 Let $(E,\theta)$ be a Higgs bundle. A holomorphic symmetric pairing $C$ of $E$ is called a symmetric pairing of $(E,\theta)$ if $C$ satisfies $C(\theta\otimes\id)= C(\id\otimes\theta)$. If the symmetric pairing~$C$ is non-degenerate on any fibers, then $C$ is called non-degenerate.
\end{Definition}

Let us recall the notion of compatibility of a non-degenerate symmetric pairing and a Hermitian metric of a $\C$-vector space of finite dimension. Let $V$ be a $\C$-vector space of finite dimension. A non-degenerate symmetric pairing $C$ and a Hermitian metric $h$ of $V$ induce the isomorphism $\Phi_C\colon V\to V^{*}$, $v\mapsto (w\mapsto C(w,v))$ and the antilinear map $\Phi_h\colon V\to V^{*}$, $v\mapsto (w\mapsto h(w,v))$, respectively. For any non-degenerate symmetric pairing $V$, we write $C^{*}$ for the symmetric pairing of $V^{*}$ induced by $C$. For any Hermitian metric $h$ of $V$, we write $h^{*}$ for the Hermitian metric of~$V^{*}$ induced by $h$.
\begin{Definition}
 Let $C$ and $h$ be a non-degenerate symmetric pairing and a Hermitian metric of $V$, respectively. The Hermitian metric $h$ is said to be compatible with $C$ if the isomorphism $\Phi_C\colon V\to V^{*}$ is an isometry with respect to $h$ and $h^{*}$.
\end{Definition}
Suppose that a Hermitian metric $h$ is compatible with a non-degenerate symmetric pairing~$C$, a real structure on $V$, that is, an antilinear involution $V\to V$ is induced as follows: $h^{*}$ induces the antilinear map $\Phi_{h^{*}}\colon V^{*}\to V^{**}=V$. Then $\kappa \coloneqq \Phi_{h^{*}}\circ \Phi_{C}\colon V\to V$ is an antilinear map. By the compatibility, we obtain $\Phi_{C^*}\circ \Phi_{h} = \Phi_{h^*}\circ \Phi_{C}$. Thus, ${\kappa}^{-1} = \Phi_{C^*}\circ \Phi_{h} = \kappa$ holds, and $\kappa$ gives a real structure on $V$.

\begin{Definition}
 Let $(E,\theta,h)$ be a harmonic bundle on $X$. A non-degenerate symmetric pairing~$C$ of $(E,\theta)$ is called a real structure if $h|_p$ is compatible with $C|_p$ for any $p\in X$, where~$h|_p$ and $C|_p$ denote the Hermitian metric and the symmetric pairing of the fiber~$E|_p$ induced by~$h$ and $C$, respectively. In this case, $h$ is said to be compatible with $C$.
\end{Definition}

\subsubsection{Symmetric pairings on filtered bundles}
\begin{Definition}
 Let $\mathcal{P}_{*}E$ be a filtered bundle on $(X,D)$. A symmetric pairing $C$ of $\mathcal{P}_{*}E$ is a~morphism $C\colon \mathcal{P}_{*}E\otimes \mathcal{P}_{*}E\to \mathcal{P}_{*} \bigl(\strsheaf_{X}^{(0)} (*D)\bigr)$ which is symmetric.
\end{Definition}
A symmetric pairing $C$ on $\mathcal{P}_{*}E$ induces a morphism $\Phi_C\colon \mathcal{P}_{*}E\to \mathcal{P}_{*}E^{*}$.
\begin{Definition}
 A symmetric pairing $C$ of a filtered bundle $\mathcal{P}_{*}E$ is called perfect if the morphism $\Phi_C\colon \mathcal{P}_{*}E\to\mathcal{P}_{*}E^{*}$ is isomorphism.
\end{Definition}
We define the notion of symmetric pairing of a filtered Higgs bundle.
\begin{Definition}
 Let $(\mathcal{P}_{*}E,\theta)$ be a good filtered Higgs bundle. A symmetric pairing~$C$ of~$(\mathcal{P}_{*}E,\theta)$ is a symmetric pairing of $\mathcal{P}_{*}E$ such that $C(\theta\otimes\id)=C(\id\otimes\theta)$.
\end{Definition}

We review the notion of wildness of harmonic bundles by following \cite{Mo1}.
\begin{Definition}
 Let $(E,\theta,h)$ be a harmonic bundle on $X\setminus D$. For $p\in D$, the morphism $f_p\colon E|_{U_p\setminus\{p\}}\to E|_{U_p \setminus\{p\}}$ is defined by $\theta = f_p {\rm d}z_p/z_p$, where $(U_p,z_p)$ is a complex chart centered at~$p$. The harmonic bundle $(E,\theta,h)$ is called wild on $(X,D)$ if all the coefficients of the characteristic polynomial $\det(t\,\id- f_p)$ are meromorphic on $U_p$ for any $p\in D$.
\end{Definition}
Let $(E,\theta,h)$ be a wild harmonic bundle on $(X,D)$. Then a filtered Higgs bundle $\bigl(\mathcal{P}_{*}^{h} E,\theta\bigr)$ is induced by $(E,\theta,h)$ as follows: for $a\in\R$ and $p\in D$,
\begin{equation}\label{induced_wild_harmonic_bundle}
 \mathcal{P}_{a}^{h} E_p = \{s\in \iota_{*}(E)_p \mid |s|_h = O(|z_p|^{-a-\varepsilon})\ \text{for any $\varepsilon>0$} \},
\end{equation}
where $\iota\colon X\setminus D\to X$ is the inclusion and $z_p$ is a holomorphic coordinate centered at $p$. In~fact, $\bigl(\mathcal{P}_{*}^{h} E,\theta\bigr)$ is a good filtered Higgs bundle \cite{Mo1}. The following theorem is a generalization of Theorem~\ref{Hitchin_thm1} to the case of wild harmonic bundles.
\begin{Theorem}[Simpson \cite{Si2}, Biquard--Boalch \cite{BB} and Mochizuki \cite{Mo1}]\label{Hitchin_thm2}
 Suppose that $X$ is compact. For a wild harmonic bundle $(E,\theta,h)$ on $(X,D)$, the induced filtered Higgs bundle $\bigl(\mathcal{P}_{*}^{h} E,\theta\bigr)$ is polystable and of degree $0$. Conversely, if a good filtered Higgs bundle $(\mathcal{P}_{*} E,\theta)$ on $(X,D)$ is polystable and $\deg{\mathcal{P}_{*}E}=0$, then the Higgs bundle $\bigl(E|_{X\setminus D},\theta\bigr)$ on $X\setminus D$ has a harmonic metric such that $\mathcal{P}_{*}^h \bigl(E|_{X\setminus D}\bigr) =\mathcal{P}_{*}E $. Moreover, if $h_1$ and $h_2$ are harmonic metrics of $(E,\theta)$ which satisfy $\mathcal{P}_{*}^{h_1} E = \mathcal{P}_{*}^{h_2} E$, then there is a decomposition as in Theorem $\ref{Hitchin_thm1}$.
\end{Theorem}

\begin{Lemma}[{\cite[Lemma 3.16]{LM1}}]\label{lemma_real_str}
 Let $(E,\theta,h)$ be a wild harmonic bundle. If $(E,\theta,h)$ has a real structure $C$, then $C$ induces a perfect symmetric pairing of the filtered Higgs bundle $\bigl(\mathcal{P}_{*}^{h} E,\theta\bigr)$.
\end{Lemma}
Suppose that $X$ is compact. By Lemma \ref{lemma_real_str} and Theorem \ref{Hitchin_thm2}, we see that if a wild harmonic bundle $(E,\theta,h)$ on $(X,D)$ has a real structure $C$, $(E,\theta,h)$ induces the good polystable filtered Higgs bundle $\bigl(\mathcal{P}_{*}^{h} E,\theta\bigr)$ of degree $0$ and $\bigl(\mathcal{P}_{*}^{h} E,\theta\bigr)$ has the induced perfect symmetric pairing. In~fact, the converse is true, that is, the following holds.
\begin{Theorem}[{\cite[Theorem 3.28]{LM1}}]\label{correspond_thm}
 The following two objects are equivalent:
 \begin{itemize}\itemsep=0pt
 \item Wild harmonic bundles on $(X,D)$ with a real structure.
 \item Good polystable filtered Higgs bundles on $(X,D)$ of degree $0$ with a perfect symmetric pairing.
 \end{itemize}
\end{Theorem}

\section{Main results}
We study harmonic metrics for $(\mathbb{K}_{\C,3},\theta(q))$ and $(\mathbb{K}_{\C^*,3},\theta(q))$ which are not generically regular semisimple. We will prove the followings in this section.
\begin{Theorem}\label{main'}
 Let $\Sigma_{\C,q}$ denote the spectral curve of $(\mathbb{K}_{\C,3},\theta(q))$. If the natural projection $\pi\colon \Sigma_{\C,q}\to \C$ is a one- or two-sheeted branched covering, then there exists a polynomial $f\in \C[z]$ such~that
 \begin{equation*}
 q_2 = 3\cdot 2^{-5/3}f^2 ({\rm d}z)^2, \qquad q_3 = f^3 ({\rm d}z)^3.
 \end{equation*}
 If $\deg{f}\geq 2$, then there exists a harmonic metric of $(\mathbb{K}_{\C,3},\theta(q))$ compatible with $C_{\C,3}$. Otherwise, there does not exist any harmonic metric compatible with $C_{\C,3}$.
\end{Theorem}

\begin{Theorem}\label{main_Cstar'}
 Let $\Sigma_{\C^*,q}$ denote the spectral curve of $(\mathbb{K}_{\C^*,3},\theta(q))$.
 If the natural projection $\pi\colon \Sigma_{\C^*,q}\to \C^*$ is a one- or two-sheeted branched covering, then there exists $f\in \C\bigl[z,z^{-1}\bigr]$ such that
 \begin{equation*}
 q_2 = 3\cdot 2^{-5/3}f^2 ({\rm d}z/z)^2, \qquad q_3 = f^3 ({\rm d}z/z)^3.
 \end{equation*}
 Unless $f$ is constant, there exists a harmonic metric of $(\mathbb{K}_{\C^*,3},\theta(q))$ compatible with $C_{\C^*,3}$.
\end{Theorem}

\subsection{Existence of harmonic bundles}\label{subsection_existence}
We consider $(\mathbb{K}_{X,3},\theta(q))$ for a generalization to other Riemann surfaces.
Let $\overline{X}$ be a compact Riemann surface and $D$ be a finite subset of $\overline{X}$. Let $X=\overline{X}\setminus D$.
We hereafter write ${(E,\theta) = (\mathbb{K}_{X,3},\theta(q))}$ and $C=C_{X,3}$.

Let $\Sigma_{E,\theta}$ be the spectral curve of $(E,\theta)$. Suppose that the natural projection $\Sigma_{E,\theta}\to X$ is a~one- or two-sheeted branched covering.
For $p\in\overline{X}$, let $\bigl(U_p,z_p\bigr)$ be a complex chart centered at $p$. In terms of local frame $({\rm d}z_p,1,({\rm d}z_p)^{-1})$, we can write
\begin{equation}\label{matrix}
 \theta \bigl({\rm d}z_p,1,\bigl({\rm d}z_p\bigr)^{-1}\bigr) = \bigl({\rm d}z_p,1,\bigl({\rm d}z_p\bigr)^{-1}\bigr) \begin{pmatrix}
 0 & q_{2,p} & q_{3,p}\\
 1 & 0 & q_{2,p} \\
 0 & 1 & 0
 \end{pmatrix} {\rm d}z_p,
\end{equation}
where $q_{2,p}$ and $q_{3,p}$ are meromorphic functions on $U_p$ which possibly have poles at $p$.
\begin{Lemma}\label{meromorphic_form}
 There exists a meromorphic $1$-form $\omega$ on $\overline{X}$ such that
 \begin{equation*}
 q_2=3\cdot 2^{-5/3}\omega^2,\qquad q_3=\omega^3.
 \end{equation*}
 Moreover, the spectral curve $\Sigma_{E,\theta}\subset K_X$ is defined as the image of the sections
 \begin{equation*}
 \lambda_1\coloneqq2^{2/3}\omega,\qquad\lambda_2\coloneqq -2^{-1/3}\omega.
 \end{equation*}
 In particular, the projection $\Sigma_{E,\theta}\to X$ is a one-sheeted branched covering if and only if $\omega\equiv 0$.
\end{Lemma}
\begin{proof}
 Since the characteristic polynomial of the matrix (\ref{matrix}) has a multiple root, we obtain $-4(-2q_{2,p})^3-27(q_{3,p})^2=0$ by considering the discriminant. Thus, there exists a meromorphic function $f_p$ on $U_p$ such that
 \begin{equation*}
 q_{2,p}=3\cdot 2^{-5/3}f_p^2,\qquad q_{3,p}=f_p^3.
 \end{equation*}
 Moreover, there is a meromorphic 1-form $\omega$ on $\overline{X}$ such that $\omega=f_p {\rm d}z_p$ for all $p\in \overline{X}$. Then ${q_2=3\cdot 2^{-5/3}\omega^2}$ and $q_3=\omega^3$. The spectral curve $\Sigma_{E,\theta}$ is given by calculation of the eigenvalues~of~(\ref{matrix}).
\end{proof}

We assume $\omega \not\equiv 0$.
For $p\in \overline{X}$, we define
\begin{equation*}
 \ord_p(\omega) = \begin{cases}
 0 & \text{if $f_p$ is holomorphic on $U_p$ and $f_p(p)\neq 0$,} \\
 k & \text{if $f_p$ has a zero of order $k$ at $p$,} \\
 -k & \text{if $f_p$ has a pole of order $k$ at $p$.}
 \end{cases}
\end{equation*}

Let $s^{(p)}_1$, $s^{(p)}_2$ and $s^{(p)}_3$ be local sections on $U_p$ defined as follows: with respect to the frame \smash{$\bigl({\rm d}z_p,1,\bigl({\rm d}z_p\bigr)^{-1}\bigr)$},
\begin{equation}\label{sections}
 s^{(p)}_1 = \begin{pmatrix}
 5\cdot 2^{-5/3}f_p^2 \\
 2^{2/3} f_p \\
 1
 \end{pmatrix},\qquad
 s^{(p)}_2 = \begin{pmatrix}
 -2^{-5/3}f_p^2 \\
 -2^{-1/3} f_p \\
 1
 \end{pmatrix},\qquad
 s^{(p)}_3 = \begin{pmatrix}
 -2^{2/3}f_p \\
 1 \\
 0
 \end{pmatrix}.
\end{equation}
Then we obtain
\begin{equation*}
 \theta\bigl(s_1^{(p)},s_2^{(p)},s_3^{(p)}\bigr) = \bigl(s_1^{(p)},s_2^{(p)},s_3^{(p)}\bigr)\begin{pmatrix}
 \lambda_1 & 0 & 0\\
 0 & \lambda_2 & {\rm d}z_p \\
 0 & 0 & \lambda_2
 \end{pmatrix}.
\end{equation*}
We define a locally free $\strsheaf_{\overline{X}}(*D)$-module $\tilde{E}$ such that $\tilde{E}|_{U_p} = \mathcal{O}_{U_p}(*p)s^{(p)}_1+\mathcal{O}_{U_p}(*p)s^{(p)}_2+\smash{\mathcal{O}_{U_p}(*p)s^{(p)}_3}$ for $p\in D$. Then the Higgs field $\theta$ and the pairing $C$ extend to the morphism $\Tilde{\theta}\colon \Tilde{E}\to\Tilde{E}\otimes K_{\overline{X}}$ and the pairing $\Tilde{C}\colon \Tilde{E}\otimes \Tilde{E}\to \strsheaf_{\overline{X}}(*D)$, respectively. Let $E_1$, $E_2$ and $E_3$ be~subbundles of $\Tilde{E}$ defined as follows:
\begin{equation*}
 E_1 = \Ker\bigl(\Tilde{\theta} -\lambda_1 \id_{\Tilde{E}}\bigr), \qquad E_2 = \Ker\bigl(\Tilde{\theta} -\lambda_2 \id_{\Tilde{E}}\bigr), \qquad E_3 = \Ker\bigl(\Tilde{\theta} - \lambda_2 \id_{\Tilde{E}}\bigr)^2.
\end{equation*}
Moreover, let $E_4$ be the smallest subbundle of $\Tilde{E}$ such that $E_1\subset E_4$ and $E_2\subset E_4$.

\begin{Lemma}
 All nontrivial Higgs subbundles of the meromorphic Higgs bundle $\bigl(\Tilde{E},\Tilde{\theta}\bigr)$ are only $E_1$, $E_2$, $E_3$ and $E_4$.
\end{Lemma}
\begin{proof}
 Let $U= \overline{X}\setminus \{p\in X\mid \omega_p=0 \}$. There exists a decomposition
 \begin{equation*}
 \bigl(\Tilde{E},\Tilde{\theta}\bigr)\big|_{U} = (E_1,\theta_1)|_{U}\oplus (E_3,\theta_3)|_{U}.
 \end{equation*}
 If $E'$ is a nontrivial Higgs subbundle of $\bigl(\Tilde{E},\Tilde{\theta}\bigr)$, then $\bigl(E'\cap E_1\bigr)|_{U}$ is a subbundle of $E_1|_{U}$ and $\bigl(E'\cap E_3\bigr)|_{U}$ is a subbundle of $E_3|_{U}$. The ranks of $\bigl(E'\cap E_1\bigr)|_{U}$ and $\bigl(E'\cap E_3\bigr)|_{U}$ completely determine $E'$, and $E'$ is one of $E_1$, $E_2$, $E_3$ or $E_4$.
\end{proof}

Suppose that $\mathcal{P}_*\Tilde{E}$ is a filtered bundle over $\Tilde{E}$ and $\Tilde{C}$ is a perfect symmetric pairing of $\bigl(\mathcal{P}_*\Tilde{E},\Tilde{\theta}\bigr)$.~The filtered bundle $\mathcal{P}_*E_i$ over $E_i$ is induced by $\mathcal{P}_*E_i = \mathcal{P}_*\Tilde{E} \cap E_i$. We assume the following.
\begin{Assumption}\label{assump}
 For any $p\in D$, there exists a decomposition
 \begin{equation*}
 \mathcal{P}_*\bigl(\Tilde{E}_p\bigr) = \mathcal{P}_*\bigl(E_{1,p}\bigr) \oplus \mathcal{P}_*\bigl(E_{3,p}\bigr).
 \end{equation*}
\end{Assumption}

\begin{Lemma}\label{pole_condition}
 Suppose that $\omega$ has a pole at $p\in D$. If the filtered Higgs bundle $\bigl(\mathcal{P}_*\Tilde{E},\Tilde{\theta}\bigr)|_{U_p}$ is good, then there exists a decomposition
 \begin{equation*}
 \mathcal{P}_*\bigl(\Tilde{E}_p\bigr) = \mathcal{P}_*\bigl(E_{1,p}\bigr) \oplus \mathcal{P}_*\bigl(E_{3,p}\bigr).
 \end{equation*}
 In particular, if any point of $D$ is a pole of $\omega$ and $\bigl(\mathcal{P}_*\Tilde{E},\Tilde{\theta}\bigr)$ is good, then the filtered bundle $\mathcal{P}_*\Tilde{E}$ satisfies Assumption~$\ref{assump}$.
\end{Lemma}
\begin{proof}
 If $\ord_p(\omega) \geq 2$, then we obtain the decomposition by definition of good filtered Higgs bundle. Suppose that $\ord_p(\omega) = 1$. We consider the vector space $\mathcal{P}_{0}\bigl(\Tilde{E}\bigr)\big|_{p}$ denoted by $V$. We have the filtration $F_{\ast}V$ on $V$ defined by \smash{$F_{a}V = \Image \bigl(\mathcal{P}_a\bigl(\Tilde{E}_p\bigr)\to V\bigr)$} for $a\in [-1,0]$. Let $\psi$ be the morphism such that $\theta = \psi {\rm d}w/w$ and $\mathrm{Res}(\psi)\colon V\to V$ be the linear map induced by $\psi$. We see that $\mathrm{Res}(\psi)(F_{a}V)\subset F_{a}V$ for any $a\in [-1,0]$. Thus, we have the generalized eigenspace decomposition of $\mathrm{Res}(\psi)$ compatible with the filtration $F_{\ast}V$. It implies the decomposition that we want to prove.
\end{proof}

\begin{Lemma}\label{lemma_E1}
 Under Assumption~$\ref{assump}$, the following holds. For $p\in D$ and \smash{$c=\bigl(c^{(p)}\bigr)_{p\in D}\in \R^{D}$},
 \begin{equation*}
 \mathcal{P}_{c}E_{1}|_{U_p} = \strsheaf_{U_p}\bigl(\bigl[c^{(p)} + \ord_{p}(\omega)\bigr]\infty\bigr) s_1^{(p)},
 \end{equation*}
 where for $a\in\R$ we write $[a]$ for the integer satisfying $a-1<[a]\leq a$.
\end{Lemma}
\begin{proof}
 Since the local section $s_1^{(p)}$ is a frame of $E_1$ on $U_p$, then there exists $d^{(p)}\in\R$ such that~\smash{$\mathcal{P}_{c}E_{1}|_{U_p} = \strsheaf_{U_p}\bigl(\bigl[c^{(p)}- d^{(p)}\bigr]\infty\bigr) s_1^{(p)}$}. The decomposition in Assumption \ref{assump} is orthogonal with respect to $\Tilde{C}$ and the induced pairing \smash{$\Tilde{C}|_{U_p}\colon \mathcal{P}_*E_1|_{U_p}\otimes \mathcal{P}_*E_1|_{U_p}\to \mathcal{P}^{(0)}_*\bigl(\mathcal{O}_{\overline{X}}(*D)\bigr)|_{U_p}$} is perfect. Therefore, we obtain
\begin{equation*}
 2d^{(p)} = -\ord_{p}{C\bigl(s_1^{(p)},s_1^{(p)}\bigr)} = -2\ord_{p}(\omega).\tag*{\qed}
 \end{equation*}
 \renewcommand{\qed}{}
\end{proof}

\begin{Lemma}\label{deg1}
 Under Assumption~$\ref{assump}$, the following holds:
 \begin{equation*}
 \deg{\mathcal{P}_*(E_1)} = \deg{\mathcal{P}_*(E_3)} = -\sum_{p\in X}\ord_p(\omega).
 \end{equation*}
\end{Lemma}
\begin{proof}
 The induced pairing
 \[
 \Tilde{C}|_{U_p}\colon \mathcal{P}_*E_1|_{U_p}\otimes \mathcal{P}_*E_1|_{U_p}\to \mathcal{P}^{(0)}_*\bigl(\mathcal{O}_{\overline{X}}(*D)\bigr)\big|_{U_p}
\]
  is perfect. Since \smash{$\ord_p\bigl(C\bigl(s_1^{(p)},s_1^{(p)}\bigr)\bigr)=2\ord_p(\omega)$} holds for $p\in X$, we obtain
 \[
 \mathcal{P}_*E_1\otimes \mathcal{P}_*E_1\cong \mathcal{P}^{(0)}_*\bigl(\mathcal{O}_{\overline{X}}(*D)\bigr)\otimes \strsheaf_{\overline{X}}\biggl(\sum_{p\in X} -2\ord_p(\omega)p\biggr).
 \]
 Thus, we see that $\deg(\mathcal{P}_*(E_1)) = -\sum_{p\in X}\ord_p(\omega)$.

 The induced pairing \[
 \Tilde{C}|_{U_p}\colon \ \mathcal{P}_*E_3|_{U_p}\otimes \mathcal{P}_*E_3|_{U_p}\to \mathcal{P}^{(0)}_*\bigl(\mathcal{O}_{\overline{X}}(*D)\bigr)\big|_{U_p}
 \]
 is perfect. Since
\smash{$\ord_p\bigl(C\bigl(s^{(p)}_2,s^{(p)}_3\bigr)\bigr)=\ord_p(\omega)$} and \smash{$C\bigl(s^{(p)}_2,s^{(p)}_2\bigr)=0$}
 hold for $p\in X$, we obtain \[
 \det(\mathcal{P}_*E_3)\otimes \det(\mathcal{P}_*E_3)\cong \smash{\mathcal{P}^{(0)}_*}\bigl(\mathcal{O}_{\overline{X}}(*D)\bigr)\otimes \strsheaf_{\overline{X}}\biggl(\sum_{p\in X} -2\ord_p(\omega)p\biggr).
 \]
 Thus, we see that
 $\deg(\mathcal{P}_*(E_3)) = \deg(\det(\mathcal{P}_*(E_3))) = -\sum_{p\in X}\ord_p(\omega)$.
\end{proof}

The holomorphic differential $q_2$ has a zero at $p\in X$ if and only if $\omega$ has a zero at $p\in X$. Thus, if $q_2$ has a zero in $X$, then $\deg{\mathcal{P}_*(E_1)}$ and $\deg{\mathcal{P}_*(E_3)}$ are negative.

\begin{Lemma}\label{polystable}
 Suppose that $q_2$ has some zeros and Assumption~$\ref{assump}$ holds. The following are equivalent:
 \begin{itemize}\itemsep=0pt
 \item $\bigl(\mathcal{P}_*\bigl(\Tilde{E}\bigr),\Tilde{\theta}\bigr)$ is polystable,
 \item $\bigl(\mathcal{P}_*\bigl(\Tilde{E}\bigr),\Tilde{\theta}\bigr)$ is stable,
 \item $\deg{\mathcal{P}_{*}E_2}$ and $\deg{\mathcal{P}_{*}E_4}$ are negative.
 \end{itemize}
\end{Lemma}
\begin{proof}
 By Lemma \ref{deg1}, $\bigl(\mathcal{P}_*\bigl(\Tilde{E}\bigr),\Tilde{\theta}\bigr)$ is stable if and only if $\deg{\mathcal{P}_{*}E_2}$ and $\deg{\mathcal{P}_{*}E_4}$ are negative.
 Short exact sequences
 \begin{align*}
 0\longrightarrow E_2\longrightarrow \Tilde{E} \longrightarrow \Tilde{E}/E_2\longrightarrow 0,\qquad
 0\longrightarrow E_4\longrightarrow \Tilde{E} \longrightarrow \Tilde{E}/E_4\longrightarrow 0
 \end{align*}
 do not split as Higgs bundles. Thus, $\bigl(\mathcal{P}_*\bigl(\Tilde{E}\bigr),\Tilde{\theta}\bigr)$ is polystable if and only if it is stable.
\end{proof}

Let us construct a filtered bundle over $\Tilde{E}$. If a filtered bundle $\mathcal{P}_{*}\Tilde{E}$ over $\Tilde{E}$ satisfies Assumption~\ref{assump} and $\Tilde{C}$ is a perfect pairing of $\mathcal{P}_{*}\Tilde{E}$, then the induced filtration $\mathcal{P}_*(E_1)$ is uniquely determined by Lemma~\ref{lemma_E1}. Thus, it suffices to consider the filtration $\mathcal{P}_*(E_3)$.
Let $\varphi = \Tilde{\theta}|_{E_3}-\lambda_2 \id_{E_3}$. $\varphi\colon E_3\to E_3 \otimes K_{\overline{X}}$ satisfies $\varphi_p\neq 0$ for $p\in X$ and $\varphi^2=0$. Therefore, $\varphi$ induces the isomorphism%
\begin{equation}\label{isomorphism}
 E_3/\Ker{\varphi} \longrightarrow \Ker{\varphi}\otimes K_{\overline{X}}.
\end{equation}
For $p\in D$, the local section \smash{$s^{(p)}_2$} is a frame of $E_2=\Ker{\varphi}$ on $U_p$ and \smash{$s^{(p)}_3$} satisfies \smash{$\varphi\bigl(s^{(p)}_3\bigr)\! = s^{(p)}_2 {\rm d}z_p$}.
\begin{Lemma}\label{section_C}
 For $p\in D$ and the local sections $s^{(p)}_2$ and $s^{(p)}_3$, the following holds:
 \begin{equation*}
 C\bigl(s^{(p)}_2,s^{(p)}_2\bigr) = 0,\qquad C\bigl(s^{(p)}_2,s^{(p)}_3\bigr) \neq 0.
 \end{equation*}
\end{Lemma}
\begin{proof}
 By (\ref{sections}), we can calculate $C\bigl(s^{(p)}_2,s^{(p)}_2\bigr)$ and $C\bigl(s^{(p)}_2,s^{(p)}_3\bigr)$.
\end{proof}

Let $v^{(p)}_2 = s^{(p)}_2$ and
\begin{equation*}
 v^{(p)}_3 = s^{(p)}_3-\frac{2C\bigl(s^{(p)}_3,s^{(p)}_3\bigr)}{C\bigl(s^{(p)}_2,s^{(p)}_3\bigr)}s^{(p)}_2.
\end{equation*}
Then \smash{$v^{(p)}_2$}, \smash{$v^{(p)}_3$} are a frame of $E_3$ on $U_p$. Moreover, by Lemma \ref{section_C}, we obtain \smash{${\varphi\bigl(v^{(p)}_3\bigr)\! = v^{(p)}_2 {\rm d}z_p}$} and \smash{$C\bigl(v^{(p)}_i,v^{(p)}_i\bigr)=0$} for $i=2,3$. For $d=(d_2,d_3)\in \R^D \times \R^D$, we define a filtered bundle $\mathcal{P}^d_{*}E_3$ over $E_3$ as follows:
\begin{equation*}
 \mathcal{P}^d_{c}E_3|_{U_p} = \mathcal{O}_{U_p}\bigl(\bigl[c^{(p)}-d^{(p)}_2\bigr]\infty\bigr)v^{(p)}_2 \oplus \mathcal{O}_{U_p}\bigl(\bigl[c^{(p)}-d^{(p)}_3\bigr]\infty\bigr)v^{(p)}_3,
\end{equation*}
where \smash{$c=(c^{(p)})_{p\in D}$} is contained in $\R^{D}$, and for $a\in\R$, $[a]$ is the integer satisfying $a-1<[a]\leq a$.
Then by $\mathcal{P}^d_{*}E_3$ and $\mathcal{P}_{*}E_1$ in Lemma \ref{lemma_E1}, we obtain the filtered bundle $\mathcal{P}_{*}^d \Tilde{E}$ which satisfies Assumption \ref{assump}.
The filtered bundles over $E_2$ and $E_3/E_2$ then are induced by
\begin{equation*}
 \mathcal{P}^{d}_{*}E_2 = \bigl(\mathcal{P}^{d}_{*}E_3\bigr) \cap E_2,\qquad \mathcal{P}^{d}_{*}(E_3/E_2) = \Image\bigl(\mathcal{P}^{d}_{*}E_3 \to E_3/E_2\bigr),
\end{equation*}
and $\deg{\mathcal{P}^{d}_*E_3} = \deg{\mathcal{P}^{d}_*E_2} + \deg{\mathcal{P}^{d}_*(E_3/E_2)}$ holds by Lemma \ref{degree_proof}.

\begin{Lemma}\label{good}
 The filtered Higgs bundle $\bigl(\mathcal{P}^d_*E_3,{\Tilde{\theta}}|_{E_3}\bigr)$ is good if and only if $d^{(p)}_2-1\leq d^{(p)}_3$ for all $p\in D$.
\end{Lemma}
\begin{proof}
 Suppose that \smash{$\bigl(\mathcal{P}^d_*E_3,{\Tilde{\theta}}|_{E_3}\bigr)$} is good. Then $\varphi=\Tilde{\theta}|_{E_3}-\lambda_2\id_{E_3}$ is logarithmic with respect to the filtrations. For $p\in D$, we obtain \smash{$d^{(p)}_2-1\leq d^{(p)}_3$} because \smash{$\varphi\bigl(v_3^{(p)}\bigr) = v_2^{(p)}{\rm d}z_p$}. Conversely, if~\smash{$d^{(p)}_2-1\leq d^{(p)}_3$} for $p\in D$, we see that $\varphi\bigl(\mathcal{P}^d_{c} E_{3,p}\bigr)\subset \mathcal{P}^d_{c+1} E_{3,p}\otimes K_{\overline{X},p}$ for $c\in\R$.
\end{proof}

\begin{Lemma}\label{perfect}
 The pairing $\Tilde{C}\colon \mathcal{P}^d_* \Tilde{E}\otimes \mathcal{P}^d_* \Tilde{E}\to \mathcal{P}_*\bigl(\mathcal{O}_{\overline{X}}(*D)\bigr)$ is perfect if and only if $d^{(p)}_2 + d^{(p)}_3 = -\ord_p(\omega)$ for any $p\in D$.
\end{Lemma}
\begin{proof}
 The induced pairing \smash{$\Tilde{C}|_{U_p}\colon \mathcal{P}_*E_3|_{U_p}\otimes \mathcal{P}_*E_3|_{U_p}\to \mathcal{P}^{(0)}_* \bigl(\mathcal{O}_{\overline{X}} (*D)\bigr)\big|_{U_p}$} is perfect for any ${p\in D}$ if and only if $\Tilde{C}$ is perfect because the decomposition in Assumption \ref{assump} is orthogonal~with~respect to $\Tilde{C}$. By \smash{$C\bigl(v^{(p)}_i,v^{(p)}_i\bigr)=0$} for $i=2,3$, the pairing $\Tilde{C}|_{U_p}$ is perfect if and~only~if%
\begin{equation*}
 d^{(p)}_2 + d^{(p)}_3 = -\ord_p C\bigl(v^{(p)}_2,v^{(p)}_3\bigr) = -\ord_p(\omega).\tag*{\qed}
 \end{equation*}
 \renewcommand{\qed}{}
\end{proof}

\begin{Lemma}\label{deg_E2}
 The following holds:
 \begin{equation*}
 \deg{\mathcal{P}^d_{*}E_2} = \frac{1}{2}\biggl(-\sum_{p\in X}\ord_p(\omega) + \sum_{p\in D} \bigl(d^{(p)}_3 - d^{(p)}_2\bigr) + 2-2g\biggr),
 \end{equation*}
where $g$ is the genus of $\overline{X}$.
\end{Lemma}
\begin{proof}
 Since the morphism ${P}^d_{d_3}(E_3/E_2)\to {P}^d_{d_2}(E_2)\otimes K_{\overline{X}}$ induced by (\ref{isomorphism}) is an isomorphism, we obtain
\begin{gather}
 \deg{\mathcal{P}^d_{*}(E_3/E_2)} + \sum_{p\in D}\sum_{d^{(p)}_{3}-1< c^{(p)}\leq d^{(p)}_{3}} c^{(p)} \dim \gr^{\mathcal{P}^d}_{c^{(p)}} (E_3/E_2)_{p} \nonumber\\
 \qquad{}= \deg{\mathcal{P}^d_{d_3}(E_3/E_2)} = \deg{\mathcal{P}^d_{d_2}E_2} + \deg{K_{\overline{X}}} \nonumber\\
 \qquad{}= \deg{\mathcal{P}^d_{*}E_2} + \sum_{p\in D}\sum_{d^{(p)}_{2}-1< c^{(p)}\leq d^{(p)}_{2}} c^{(p)} \dim \gr^{\mathcal{P}^d}_{c^{(p)}} \bigl(E_{2,p}\bigr) + 2g -2,\label{deg2}
\end{gather}
 Since
\begin{gather*}
 \sum_{d^{(p)}_{3}-1< c^{(p)}\leq d^{(p)}_{3}} c^{(p)} \dim \gr^{\mathcal{P}^d}_{c^{(p)}} (E_3/E_2)_{p} = d^{(p)}_3, \qquad
 \sum_{d^{(p)}_{2}-1< c^{(p)}\leq d^{(p)}_{2}} c^{(p)} \dim \gr^{\mathcal{P}^d}_{c^{(p)}} E_{2,p} = d^{(p)}_2,
\end{gather*}
we obtain $\deg{\mathcal{P}^d_{*}E_2} = \frac{1}{2}\bigl(-\sum_{p\in X}\ord_p(\omega) + \sum_{p\in} \bigl(d^{(p)}_3 - d^{(p)}_2\bigr) + 2-2g\bigr)$ from Lemma \ref{deg1}, ${\deg{\mathcal{P}^{d}_*E_3} = \deg{\mathcal{P}^{d}_*E_2} + \deg{\mathcal{P}^{d}_*(E_3/E_2)}}$ and (\ref{deg2}).
\end{proof}

\begin{Lemma}\label{deg_E4}
 The following holds:
 \begin{equation*}
 \deg{\mathcal{P}^d_{*}E_4} = \frac{1}{2}\biggl(-\sum_{p\in X}\ord_p(\omega) + \sum_{p\in D} \bigl(d^{(p)}_3 - d^{(p)}_2\bigr) + 2-2g\biggr).
 \end{equation*}
\end{Lemma}
\begin{proof}
 Let $\bigl[v^{(p)}_2\bigr]$ and $\bigl[v^{(p)}_3\bigr]$ be the images of $v^{(p)}_2$ and $v^{(p)}_3$ under the natural projection $\mathcal{P}^d_{*}{\Tilde{E}}\to \mathcal{P}^d_{*}\bigl(\Tilde{E}/E_1\bigr)$, respectively. The images \smash{$\bigl[v^{(p)}_2\bigr]$}
 and \smash{$\bigl[v^{(p)}_3\bigr]$} give a frame of $\Tilde{E}/E_1$ on $U_p$ such that~\smash{$\bigl[v^{(p)}_2\bigr]$} is a frame of $E_4/E_1$. Then we see that
\begin{align*}
 \deg{\mathcal{P}^d_{*}(E_4/E_1)} &{}= \frac{1}{2}\biggl(\deg{\mathcal{P}^d_{*}\bigl(\Tilde{E}/E_1\bigr)} + \sum_{p\in D} \bigl(d^{(p)}_3 - d^{(p)}_2\bigr) + 2-2g\biggr)\\
 &{}= \frac{1}{2}\biggl(\sum_{p\in X}\ord_p(\omega) + \sum_{p\in D} \bigl(d^{(p)}_3 - d^{(p)}_2\bigr) + 2-2g\biggr),
 \end{align*}
where the second equality holds from Lemmas \ref{degree_proof} and \ref{deg1}.
Considering the short exact sequence $0\to E_1\to E_4\to E_4/E_1 \to 0$, by Lemmas \ref{degree_proof} and \ref{deg1}, we obtain
\begin{align*}
 \deg{\mathcal{P}^d_{*}E_4} &{}= \deg{\mathcal{P}^d_{*}(E_4/E_1)} + \deg{\mathcal{P}^d_{*}E_1} \\
 &{}= \frac{1}{2}\biggl(-\sum_{p\in X}\ord_p(\omega) + \sum_{p\in D} \bigl(d^{(p)}_3 - d^{(p)}_2\bigr) + 2-2g\biggr).
\tag*{\qed}
\end{align*}
\renewcommand{\qed}{}
\end{proof}

\begin{Proposition}\label{main_prop}
 Suppose that $q_2$ has a zero in $X$ and $d_2, d_3\in \R^D$ satisfy the conditions that
 \begin{itemize}\itemsep=0pt
 \item $d^{(p)}_2-1\leq d^{(p)}_3$ for $p\in D$,
 \item $d^{(p)}_2 + d^{(p)}_3 = -\ord_p(\omega)$ for $p\in D$,
 \item $-\sum_{p\in X}\ord_p(\omega) + \sum_{p\in D} \bigl(d^{(p)}_3 - d^{(p)}_2\bigr) + 2-2g <0$.
 \end{itemize}
 Then the filtered Higgs bundle \smash{$\bigl(\mathcal{P}^d_{*}{\Tilde{E}},\Tilde{\theta}\bigr)$} is good and stable, and the induced symmetric pairing~$\Tilde{C}$ is perfect. Conversely, if \smash{$\bigl(\mathcal{P}^d_{*}{\Tilde{E}},\Tilde{\theta}\bigr)$} is good and polystable, and the induced pairing $\Tilde{C}$ is perfect, then $d_2$ and~$d_3$ satisfy the above conditions.
\end{Proposition}
\begin{proof}
 From Lemma \ref{good}, $\bigl(\mathcal{P}^d_{*}{\Tilde{E}},\Tilde{\theta}\bigr)$ is a good filtered Higgs bundle.
 From Lemmas~\ref{polystable},~\ref{deg_E2} and~\ref{deg_E4}, we see that $\bigl(\mathcal{P}^d_{*}{\Tilde{E}},\Tilde{\theta}\bigr)$ is stable. From Lemma \ref{perfect}, we see that $\Tilde{C}$ is perfect.
\end{proof}

The following corollaries are part of Theorems \ref{main'} and \ref{main_Cstar'}.
\begin{Corollary}\label{main_cor}
 Suppose that $\bigl(\overline{X},D\bigr)=\bigl(\ps^1,\{\infty\}\bigr)$. There exists a polynomial $f\in \C[z]$ such~that
 \begin{equation*}
 q_2 = 3\cdot 2^{-5/3}f^2 ({\rm d}z)^2, \qquad q_3 = f^3 ({\rm d}z)^3.
 \end{equation*}
 If $\deg{f}\geq 2$, then there exists a harmonic metric of $(E,\theta)$ compatible with $C$.
\end{Corollary}
\begin{proof}
 By Lemma \ref{meromorphic_form}, there exists a polynomial $f\in \C[z]$ such that
 \begin{equation*}
 q_2 = 3\cdot 2^{-5/3}f^2 ({\rm d}z)^2, \qquad q_3 = f^3 ({\rm d}z)^3.
 \end{equation*}
 Then we see that $\ord_{\infty}(\omega) = -(\deg{f} + 2 )$.
 Suppose that $d_2^{(\infty)} = (\deg{f} + 3 )/2$ and $d_3^{(\infty)} = (\deg{f} + 1 )/2$ for $p\in D$.
 By Proposition \ref{main_prop}, if $\deg{f} = \sum_{p\in X}\ord_p(\omega) \geq 2$, then $\bigl(\mathcal{P}^d_{*}{\Tilde{E}},\Tilde{\theta}\bigr)$ is a~good stable filtered Higgs bundle with the perfect symmetric pairing \smash{$\Tilde{C}$}. Therefore, if $\deg{f} \geq 2$, then there exists a harmonic metric of $(E,\theta)$ compatible with $C$ from Theorem~\ref{correspond_thm}.
\end{proof}

\begin{Corollary}\label{main_Cstar_cor}
 Suppose that $(\overline{X},D) = \bigl(\ps^1,\{0,\infty\}\bigr)$. There exists a rational function $f\in \C\bigl[z,z^{-1}\bigr]$ such that
 \begin{equation*}
 q_2 = 3\cdot 2^{-5/3}f^2 ({\rm d}z/z)^2, \qquad q_3 = f^3 ({\rm d}z/z)^3.
 \end{equation*}
 Moreover, if $f$ has a zero in $\C^*$ or $f(z)=a z^b$, $a\in\C^*$, $b\in\Z$, $|b|\geq 3$, then there exists a~harmonic metric of $(E,\theta)$ compatible with $C$.
\end{Corollary}

\begin{proof}
 By Lemma \ref{meromorphic_form}, there exists a rational function $f\in \C\bigl[z,z^{-1}\bigr]$ such that
 \begin{equation*}
 q_2 = 3\cdot 2^{-5/3}f^2 ({\rm d}z/z)^2, \qquad q_3 = f^3 ({\rm d}z/z)^3.
 \end{equation*}
 The case when $f$ is of the form $f(z)=az^b$ reduces to the case of Corollary \ref{main_cor}. Thus, there exists a harmonic metric of $(E,\theta)$ compatible with $C$. Suppose that $f$ has a~zero in~$\C^*$. Then we~obtain that $\sum_{p\in X}\ord_p(\omega)\geq 1$. By Proposition \ref{main_prop}, if we set \smash{$d^{(p)}_2$} and \smash{$d^{(p)}_3$} as $(-\ord_{p}(\omega)+1)/2$ and $(-\ord_{p}(\omega)-1)/2$, respectively, then there exists a harmonic metric of $(E,\theta)$ compatible with~$C$.
\end{proof}

\begin{Remark}\label{condition_of_d_2}
 In the case of $(\overline{X},D)=\bigl(\ps^1,\{\infty\}\bigr)$, if $2< d_2^{(\infty)} \leq (\deg{f} + 3 )/2$ and $d_3^{(\infty)} = \smash{\deg{f} +2 - d_2^{(\infty)}}$ hold, then $d_2$ and $d_3$ satisfy the conditions in Proposition \ref{main_prop}. Thus, we obtain a~family of harmonic metrics compatible with $C$ which is parameterized by \smash{$d_2^{(\infty)}$}.
\end{Remark}

The following proposition also gives part of Theorems \ref{main'} and \ref{main_Cstar'}.
\begin{Proposition}\label{remark_degree}
 For the two cases of $(\overline{X},D)=\bigl(\ps^1,\{\infty\}\bigr)$ or $\bigl(\ps^1,\{0,\infty\}\bigr)$, if $f$ is constant, then there does not exist any harmonic metric of $(E,\theta)$ compatible with $C$.
\end{Proposition}
\begin{proof}
 Suppose that $f$ is a non-zero constant function and there exists a harmonic metric~$h$ of~$(E,\theta)$ compatible with $C$. Then we obtain the good filtered Higgs bundle $\bigl(\mathcal{P}^h_{*} E,\theta\bigr)$ defined by~(\ref{induced_wild_harmonic_bundle}).
 By Lemma \ref{pole_condition}, the decomposition $(E,\theta) = (E_1,\theta')\oplus (E_3,\theta '')$ extends to the decomposition as the good filtered Higgs bundle and the induced filtered Higgs bundles $\bigl(\mathcal{P}^h_{*} E_1,\theta'\bigr)$ and $\bigl(\mathcal{P}^h_{*} E_3,\theta''\bigr)$ are polystable. Therefore, $(E_1,\theta')$ and $(E_3,\theta '')$ have harmonic metrics. However, because the Higgs field $\theta''-\lambda_2\,\id$ is nilpotent and non-zero, there does not exist any harmonic metric of $(E_3,\theta '')$ (see \cite[Propositions~3.41 and~3.42]{LM3} for details). Hence, there does not exist any harmonic metric of $(E,\theta)$ compatible with $C$.
 If $f\equiv 0$, the Higgs bundle $(E,\theta)$ does not have a harmonic metric since $\theta$ is nilpotent and non-zero.
\end{proof}

\subsection[The case of C]{The case of $\boldsymbol{\C}$}

In this subsection, we consider the case when $\bigl(\overline{X},D\bigr)=\bigl(\mathbb{P}^1,\{\infty\}\bigr)$ and $\deg{f}= 1$, where $f$ is the polynomial in Theorem~\ref{main'}.
Let $z$ be a holomorphic coordinate of $\C$ and let ${w=1/z}$. We hereafter write $(E,\theta) = (\mathbb{K}_{\C,3},\theta(q))$ and $C=C_{\C,3}$. Let $\psi$ be the morphism such that~${\theta = \psi {\rm d}w/w}$.

Let $\Tilde{E}$ be a meromorphic extension of $E$ defined as
\[
\tilde{E}|_{U_{\infty}} = \mathcal{O}_{U_{\infty}}(*\infty)s^{(\infty)}_1+\mathcal{O}_{U_{\infty}}(*\infty)s^{(\infty)}_2+\mathcal{O}_{U_{\infty}}(*\infty)s^{(\infty)}_3,
\]
where $U_{\infty}$ is a neighborhood of $\infty$ and \smash{$s^{(\infty)}_1$}, \smash{$s^{(\infty)}_2$} and \smash{$s^{(\infty)}_3$} are the sections defined as (\ref{sections}).
Let~\smash{$v^{(\infty)}_2 = s^{(\infty)}_2$} and
\begin{equation*}
 v^{(\infty)}_3 = s^{(\infty)}_3-\frac{2C\bigl(s^{(\infty)}_3,s^{(\infty)}_3\bigr)}{C\bigl(s^{(\infty)}_2,s^{(\infty)}_3\bigr)}s^{(\infty)}_2.
\end{equation*}
Then we obtain \smash{$C\bigl(v^{(\infty)}_i,v^{(\infty)}_i\bigr)=0$} for $i=2,3$.

\begin{Lemma}\label{mero_unique}
 Let $\Tilde{E'}$ be a meromorphic extension of $E$, i.e., a locally free $\strsheaf_{\ps^1}(*\infty)$-module $\Tilde{E'}$ such that $\Tilde{E'}|_{\C} = E$. Suppose that $\Tilde{E'}$ satisfies the following:
 \begin{itemize}\itemsep=0pt
 \item the pairing $\Tilde{C'}$ induced by $C$ satisfies $\Tilde{C'}\bigl(\Tilde{E'}\otimes \Tilde{E'}\bigr)\subset \strsheaf_{\ps^1}(*\infty)$,
 \item the morphism $\Tilde{\theta '}$ induced by $\theta$ satisfies $\Tilde{\theta'}\bigl(\Tilde{E'}\bigr)\subset \Tilde{E'}\otimes K_{\ps^1}$.
 \end{itemize}
 Then the meromorphic extension $\Tilde{E'}$ is equal to $\Tilde{E}$.
\end{Lemma}

\begin{proof}
 Let $\psi'\colon \Tilde{E'}_{\infty}\to\Tilde{E'}_{\infty}$ be a morphism defined by $\Tilde{\theta '}= \psi' {\rm d}w/w$. Since the characteristic polynomial $\det(t\,\id- \psi')$ is equal to $\det(t\,\id- \psi)$, $\psi$ and $\psi'$ have the same eigenvalues. Thus, there exists the decomposition $\Tilde{E'}_{\infty}=E'_{1,\infty}\oplus E'_{3,\infty}$, where ${E'_1 = \Ker\bigl(\Tilde{\theta '}-\lambda_1 \id_{\Tilde{E'}}\bigr)}$ and \smash{${E'_3 = \Ker\bigl(\Tilde{\theta '}-\lambda_2 \id_{\Tilde{E'}}\bigr)^2}$}. Let \smash{${s'}_1^{(\infty)}$} be a frame of $E'_{1}$ on a~neighborhood $U'$ of~$\infty$. There exists a~holomorphic function $h_1$ on $U'\setminus\{\infty\}$ such that \smash{${s'}_1^{(\infty)} = h_1 s_1^{(\infty)}$}. Then since $\smash{C\bigl({s'}_1^{(\infty)},{s'}_1^{(\infty)}\bigr)} = \smash{h_1^2 C\bigl(s_1^{(\infty)},s_1^{(\infty)}\bigr)}$ is meromorphic, the function $h_1$ is meromorphic on $U'$. Therefore, we obtain $E'_{1} = E_{1}$. It suffices to prove $E'_{3} = E_{3}$. Let \smash{$E'_2 = \Ker\bigl(\Tilde{\theta '}-\lambda_2 \id_{\Tilde{E'}}\bigr)$} and \smash{${s'}_2^{(\infty)}$} be a frame of $E'_{2}$ on $U'$. Since the morphism $\varphi' = \Tilde{\theta '}|_{E'_3}-\lambda_2 \id_{E'_3} \colon E'_3\to E'_3$ induces the isomorphism
 \begin{equation*}
 \varphi' \colon \ E'_3 / E'_2 \to E'_2 \otimes K_{\ps^1},
 \end{equation*}
 we can take a local section \smash{${s'}_3^{(\infty)}$} satisfying \smash{$\varphi'\bigl({s'}_3^{(\infty)}\bigr) = {s'}_2^{(\infty)}{\rm d}w$}. Because \smash{$C\bigl({s'}_2^{(\infty)},{s'}_3^{(\infty)}\bigr)\neq 0$}, we~can construct ${v'}_2^{(\infty)}$ and ${v'}_3^{(\infty)}$ in the same way that we constructed $v^{(\infty)}_2$ and $v^{(\infty)}_3$. Then we~obtain that \smash{$C\big({v'}_i^{(\infty)},{v'}_i^{(\infty)}\big)=0$} for $i=2,3$ and \smash{$\varphi\big({v'}_3^{(\infty)}\big) ={v'}_2^{(\infty)} {\rm d}w$}. There exists a holomorphic function $h_3$ on $U'\setminus\{\infty\}$ such that ${v'}_3^{(\infty)} = h_3 v_3^{(\infty)}$. Then we obtain ${v'}_2^{(\infty)} = h_3 v_2^{(\infty)}$. Since~\smash{$C\big({v'}_2^{(\infty)},{v'}_3^{(\infty)}\big)= C\big(h_3{v}_2^{(\infty)},h_3{v}_3^{(\infty)}\big)= h_3^2 C\big({v}_2^{(\infty)},{v}_3^{(\infty)}\big)$} is meromorphic on $U'$, $h_3$ is meromorphic on $U'$. Therefore, we obtain $E_3 = E'_3$.
\end{proof}

Because of Lemma \ref{mero_unique}, we focus on only filtered bundles over $\Tilde{E}$.
Let $\mathcal{P}_*\Tilde{E}$ be a filtered bundle over $\Tilde{E}$. Suppose that $\big(\mathcal{P}_*\bigl(\Tilde{E}\bigr),\Tilde{\theta}\big)$ is good and the induced pairing $\Tilde{C}$ is perfect. If $s\in \Tilde{E}_{\infty}$ satisfies that for $c\in\R$, $s\in \mathcal{P}_{c}\Tilde{E}_{\infty}$ and $s\notin \mathcal{P}_{<c}\Tilde{E}_{\infty}$, we say that
 the degree of $s$ is $c$ and we write $\deg{s}$ for the degree of $s$. We define $\deg{s}=-\infty$ if $s=0$. Let \smash{$d_2= \deg{v^{(\infty)}_2}$} and \smash{$d_3=\deg{v^{(\infty)}_3}$}.
First, we prove the following lemma.

\begin{Lemma}\label{assump_satisfy}
 The filtered Higgs bundle $\mathcal{P}_*\bigl(\Tilde{E}\bigr)$ satisfies Assumption $\ref{assump}$, that is,
 \begin{equation*}
 \mathcal{P}_*\Tilde{E}_{\infty} = \mathcal{P}_*E_{1,\infty} \oplus \mathcal{P}_*E_{3,\infty}.
 \end{equation*}
\end{Lemma}
\begin{proof}
 Since the meromorphic 1-form $\omega=f{\rm d}z$ has a pole at $\infty$, we obtain the desired decomposition by Lemma \ref{pole_condition}.
\end{proof}

 Note that $d_3-d_2\geq -1$ by Lemma \ref{good}.
\begin{Lemma}\label{case1}
 Let $U$ be a small neighborhood of $\infty$. If $d_3-d_2\notin\Z$, then
 \begin{equation*}
 \mathcal{P}_{c}E_3|_{U} = \mathcal{O}_U([c-d_2]\infty)v^{(\infty)}_2 \oplus \mathcal{O}_U([c-d_3]\infty)v^{(\infty)}_3.
 \end{equation*}
\end{Lemma}
\begin{proof}
 Let \smash{$s'_3=w^{[d_3-d_2]} v^{(\infty)}_3$} and $d'_3 = \deg{s'_3}$. Then $0<d'_3-d_2<1$. It suffices to prove \smash{$\mathcal{P}_{d_2}E_3|_{U} = \mathcal{O}_U v^{(\infty)}_2 \oplus \mathcal{O}_U ws'_3$} and \smash{$\mathcal{P}_{d'_3}E_3|_{U} = \mathcal{O}_U v^{(\infty)}_2 \oplus \mathcal{O}_U s'_3$}. Let $u_2$ and $u_3$ be local sections~of~$\mathcal{P}_{d_2}E_3$ and $\mathcal{P}_{d'_3}E_3$ on $U$, respectively, such that $u_2\neq 0$ in $\gr^{\mathcal{P}}_{d_2} (E_{3,\infty})$ and $u_3\neq 0$ in \smash{$\gr^{\mathcal{P}}_{d'_3} (E_{3,\infty})$}. Then there exist $g_i,h_i \in \mathcal{O}_U(*\infty)$ such that
 \begin{equation*}
 v^{(\infty)}_2 = g_2 u_2 +g_3 u_3, \qquad s'_3 = h_2 u_2 +h_3 u_3.
 \end{equation*}
 We obtain $\deg{(g_2 u_2)}=d_2$ and $\deg{(g_3 u_3)}\leq d'_3-1$ since $\deg{(g_2 u_2)}\in (d_2 +\Z) \cup \{-\infty\}$ and $\deg{(g_3 u_3)}\in \big(d'_3 +\Z\big) \cup \{-\infty\}$, where for $a\in \R$, we write $a+\Z$ for $\{a+n\mid n\in\Z\}$.
 Similarly, we obtain $\deg{(h_2 u_2)}\leq d_2$ and $\deg{(h_3 u_3)}=d'_3$.
 Therefore, \smash{$v^{(\infty)}_2$} and $ws'_3$ are linearly independent in $\mathcal{P}_{d_2}E_3|_{\infty}$. Similarly, \smash{$v^{(\infty)}_2$} and $s'_3$ are linearly independent in ${\mathcal{P}}_{d'_3} E_{3}|_{\infty}$.
\end{proof}

\begin{Lemma}
 Let $U$ be a small neighborhood of $\infty$. If $d_3-d_2=-1$, then
 \begin{equation*}
 \mathcal{P}_{c}E_3|_{U} = \mathcal{O}_U([c-d_2]\infty)v^{(\infty)}_2 \oplus \mathcal{O}_U([c-d_3]\infty)v^{(\infty)}_3.
 \end{equation*}
\end{Lemma}
\begin{proof}
 Suppose that $a_2wv^{(\infty)}_2+a_3v^{(\infty)}_3=0$ in $\gr^{\mathcal{P}}_{d_3}(E_{3,\infty})$ for $a_2,a_3\in\C$.
Since
 \[
 \big(\psi-2^{-1/3}w^{-1}f\big)\big(v^{(\infty)}_2\big)=0 \qquad \text{and} \qquad \big(\psi-2^{-1/3}w^{-1}f\big)\big(v^{(\infty)}_3\big)=wv^{(\infty)}_2,
 \]
 we obtain
 \begin{equation*}
 0 = \big(\psi-2^{-1/3}w^{-1}f\big)\big(a_2wv^{(\infty)}_2+a_3v^{(\infty)}_3\big) = a_3wv^{(\infty)}_2\qquad \text{in}\ \gr^{\mathcal{P}}_{d_3}(E_{3,\infty}).
 \end{equation*}
 Thus, $a_3=0$ and $a_2=0$, that is, \smash{$wv^{(\infty)}_2$} and \smash{$v^{(\infty)}_2$} are linearly independent in $\gr^{\mathcal{P}}_{d_3}(E_{3,\infty})$.
\end{proof}

\begin{Lemma}\label{case3}
 Let $U$ be a small neighborhood of $\infty$. If $d_3-d_2\in\Z_{\geq 0}$, then
 \begin{equation}\label{case3'}
 \mathcal{P}_{c}E_3|_{U} = \mathcal{O}_U([c-d_2]\infty)v^{(\infty)}_2 \oplus \mathcal{O}_U([c-d_3]\infty)v^{(\infty)}_3.
 \end{equation}
\end{Lemma}
\begin{proof}
 We take a local frame $\big(w^{n}v^{(\infty)}_2,v=av^{(\infty)}_2+bv^{(\infty)}_3\big)$ of $\mathcal{P}_{0} E_3$ on $U$, where $n\in \Z$, $a,b\in H^0(U,\strsheaf_{\ps^1}(*\infty))$. There are two cases.
 \begin{enumerate}\itemsep=0pt\setlength{\leftskip}{0.62cm}
 \item[\bf{Case 1.}] There exists $c\in\R$ such that $\dim \gr^{\mathcal{P}}_{c} E_{3,\infty}=2$.
 \item[\bf{Case 2.}] There exists $c\in\R$ such that $\dim \gr^{\mathcal{P}}_{c} E_{3,\infty}=1$.
 \end{enumerate}

 \begin{Lemma}
 In Case $1$, $\pa(\mathcal{P}_{*} E_{3,\infty}) = \Z $ or $\pa(\mathcal{P}_{*} E_{3,\infty}) = \Z + 1/2$ holds. Here, $\Z+1/2=\{n+ 1/2 \mid n\in\Z\}$.
 \end{Lemma}
 \begin{proof}
 The induced pairing
$ \Tilde{C}|_{U}\colon \mathcal{P}_{*}E_3|_{U}\otimes\mathcal{P}_{*}E_3|_{U}\to \mathcal{P}^{(0)}_{*}(\strsheaf_{U}(*\infty))
$
 is perfect by Lemma~\ref{assump_satisfy}. Therefore, for $c\in \pa(\mathcal{P}_{*} E_{3,\infty})$, the non-degenerate symmetric pairing \[
 \gr^{\mathcal{P}}_{c} (E_{3,\infty})\otimes \gr^{\mathcal{P}}_{c} (E_{3,\infty}) \to \gr^{\mathcal{P}^{(0)}}_{2c} (\strsheaf_{U,\infty})
\] is induced. Then since \smash{$\dim \gr^{\mathcal{P}^{(0)}}_{2c} (\strsheaf_{U,\infty})\neq 0$}, we see that $2c\in\Z$.
 \end{proof}
 We first consider the case of $\pa(\mathcal{P}_{*} E_{3,\infty}) = \Z$. Then the induced pairing
 \[
 \gr^{\mathcal{P}}_{0} (E_{3,\infty})\otimes \gr^{\mathcal{P}}_{0} (E_{3,\infty}) \to \gr^{\mathcal{P}^{(0)}}_{0} (\strsheaf_{U,\infty})
 \]
 is non-degenerate. We see that \smash{$C\big(w^{n}v^{(\infty)}_2,v\big) \neq 0$} in \smash{$\gr^{\mathcal{P}^{(0)}}_{0} (\strsheaf_{U,\infty})$} since \smash{$C\big(w^{n}v^{(\infty)}_2,w^{n}v^{(\infty)}_2\big) = 0$} and \smash{$\deg v=\deg\big(w^nv^{(\infty)}_2\big)=0$} hold. Then we obtain
 \begin{align*}
 &\ord_{\infty}C(v,v) = \ord_{\infty}a+\ord_{\infty}b-\deg{f}-2 \geq 0, \\
 &\ord_{\infty}C(w^n v^{(\infty)}_2,v) = \ord_{\infty} b - \deg{f}-2 + n = 0.
 \end{align*}
 Thus, we obtain $\ord_{\infty}a \geq n$, and we can take the local frame \smash{$\big(w^nv^{(\infty)}_2,bv^{(\infty)}_3\big)$} of $\mathcal{P}_{0} E_3$ on~$U$. It implies the equation (\ref{case3'}). In the case of $\pa(\mathcal{P}_{*} E_3) = \Z + 1/2$, the induced pairing \[
 \gr^{\mathcal{P}}_{-1/2} (E_{3,\infty})\otimes \gr^{\mathcal{P}}_{-1/2} (E_{3,\infty}) \to \gr^{\mathcal{P}^{(0)}}_{-1} (\strsheaf_{U,\infty})
 \]
 is non-degenerate. We see that $C\big(w^{n}v^{(\infty)}_2,v\big) \neq 0$ in \smash{$\gr^{\mathcal{P}^{(0)}}_{-1} (\strsheaf_{U,\infty})$} since \smash{$C\big(w^{n}v^{(\infty)}_2,w^{n}v^{(\infty)}_2\big) = 0$} and \smash{$\deg v=\deg\big(w^nv^{(\infty)}_2\big)=-1/2$} hold. We obtain
 \begin{align*}
 &\ord_{\infty}C(v,v) = \ord_{\infty}a+\ord_{\infty}b-\deg{f}-2 \geq 1, \\
 &\ord_{\infty}C(w^n v^{(\infty)}_2,v) = \ord_{\infty} b - \deg{f}-2 + n = 1.
 \end{align*}
 Therefore, we can take the frame $\big(w^nv^{(\infty)}_2,bv^{(\infty)}_3\big)$ of $\mathcal{P}_{0} E_3$ on U.

In Case 2, we may assume that $\deg{w^{n}v^{(\infty)}_2} \neq \deg{v}$. Let $c_2$ and $c_3$ denote $\deg{w^{n}v^{(\infty)}_2}$ and $\deg{v}$, respectively. Then $-1<c_2,c_3\leq 0$ holds. The induced pairing
 \begin{equation*}
 \gr^{\mathcal{P}}_{c_2} (E_{3,\infty})\otimes \gr^{\mathcal{P}}_{c_3} (E_{3,\infty}) \to \gr^{\mathcal{P}^{(0)}}_{c_2+c_3} (\strsheaf_{U,\infty})
 \end{equation*}
 is perfect because of \smash{$C\big(w^{n}v^{(\infty)}_2,w^{n}v^{(\infty)}_2\big) = 0$}. Thus, we obtain $c_2 + c_3= -1$. Moreover, we obtain that
 \begin{equation*}
 \ord_{\infty}C\big(w^n v^{(\infty)}_2,v\big) = \ord_{\infty} b - \deg{f}-2 + n = -(c_2+c_3) = 1.
 \end{equation*}
 \begin{Lemma}
 One of the following holds:
 \begin{itemize}\itemsep=0pt
 \item $\ord_{\infty} a > n$,
 \item $c_2 < c_3$ and $\ord_{\infty} a=n$.
 \end{itemize}
 \end{Lemma}
 \begin{proof}
 Because of $2c_3 < 0$, we obtain
 \begin{equation*}
 \ord_{\infty}C(v,v) = \ord_{\infty}a+\ord_{\infty}b-\deg{f}-2 \geq 1,
 \end{equation*}
 and it gives $\ord_{\infty} a\geq n$. Suppose $\ord_{\infty}a=n$. Then we obtain $\ord_{\infty}C(v,v)=1$. It implies that $2c_3\geq-1$, and we see that $2c_3>-1$ because $c_2\neq c_3$ and $c_2+c_3=-1$ hold. Therefore, we obtain $c_2 < c_3$.
 \end{proof}

 \begin{Lemma}\label{cannot_occur1}
 The condition that $\ord_{\infty} a > n$ cannot occur.
 \end{Lemma}
 \begin{proof}
 Suppose $\ord_{\infty} a > n$. Then we obtain that $\deg{bv^{(\infty)}_3} = \deg{v} = c_3$. Therefore, we see that $d_3-c_3\in \Z$ and $d_2-c_2\in \Z$. This contradicts $d_3-d_2\in\Z$ and $|c_3-c_2|<1$.
 \end{proof}

 \begin{Lemma}\label{cannot_occur2}
 The condition that $c_2 < c_3$ and the condition that $\ord_{\infty} a=n$ cannot occur at the same time.
 \end{Lemma}
 \begin{proof}
 Suppose that $c_2 < c_3$ and $\ord_{\infty} a=n$ hold. Then we obtain that
 \[
 \deg{bv^{(\infty)}_3} = \deg{\big(v-(aw^{-n})w^n v^{(\infty)}_2\big)} = c_3.
 \]
 This contradicts $d_3-c_2\in \Z$ and $|c_3-c_2|<1$.
 \end{proof}
 By Lemmas \ref{cannot_occur1} and \ref{cannot_occur2}, we see that Case 2 cannot occur.
\end{proof}

By Lemmas \ref{case1}--\ref{case3}, we obtain the following theorem.

\begin{Theorem}\label{classifying}
 For the good filtered Higgs bundle $\big(\mathcal{P}_*\Tilde{E},\Tilde{\theta}\big)$ with the perfect pairing $\Tilde{C}$, there exists a neighborhood $U$ of $\infty$ such that
 \begin{equation*}
 \mathcal{P}_*\Tilde{E}|_{U} = \mathcal{P}_*E_1|_{U} \oplus \mathcal{P}_*E_3|_{U}.
 \end{equation*}
 Moreover, $\mathcal{P}_*E_1|_{U}$ and $\mathcal{P}_*E_3|_{U}$ are represented as
 \begin{align*}
 & \mathcal{P}_{c}E_{1}|_{U} = \strsheaf_{U}([c -\deg{f}-2]\infty) s_1^{(\infty)}, \\
 & \mathcal{P}_{c}E_3|_{U} = \mathcal{O}_U([c-d_2]\infty)v^{(\infty)}_2 \oplus \mathcal{O}_U([c-d_3]\infty)v^{(\infty)}_3.
 \end{align*}
\end{Theorem}

\begin{Corollary}\label{sub_cor}
 If the polynomial $f\in \C[z]$ is of degree $1$, there is no harmonic metric of $(E,\theta)$ compatible with $C$.
\end{Corollary}
\begin{proof}
 If $\deg{f}=1$, by Remark \ref{condition_of_d_2} and Theorem \ref{classifying}, there does not exist any good polystable filtered Higgs bundle over $\Tilde{E}$ with a perfect pairing. By Theorem \ref{correspond_thm}, there is no harmonic metric of $(E,\theta)$ compatible with $C$.
\end{proof}

By Corollary \ref{main_cor}, \ref{sub_cor} and Proposition \ref{remark_degree}, we obtain Theorem \ref{main'}.
\begin{Remark}\label{rem_classifying}
 Theorem \ref{classifying} gives the classification of compatible harmonic metrics of $(E,\theta)$. In other words, compatible harmonic metrics of $(E,\theta)$ are given by filtered Higgs bundles in Theorem~\ref{classifying} and parameterized by $d_2$, which can take any value in the interval $(2,(\deg{f}+3)/2]$.
\end{Remark}

\subsection[The case of C*]{The case of $\boldsymbol{\C^*}$}
In this subsection, we consider the case of $\bigl(\overline{X},D\bigr)=\bigl(\mathbb{P}^1,\{0,\infty\}\bigr)$.
To complete the proof of Proposition~\ref{main_Cstar'}, it suffices to study the case of $f=az^b$, $a\in \C^*$, $b=1,2$ because of Corollary~\ref{main_Cstar_cor} and Proposition \ref{remark_degree}. Suppose that $f=az^b$.
Let $s_1$, $s_2$, $s_3$ be sections defined as follows: with respect to the frame $\big({\rm d}z/z,1,({\rm d}z/z)^{-1}\big)$,
\begin{equation*}
 s_1 = \begin{pmatrix}
 5\cdot 2^{-5/3}f^2 \\
 2^{2/3} f \\
 1
 \end{pmatrix}, \qquad
 s_2 = \begin{pmatrix}
 -2^{-5/3}f^2 \\
 -2^{-1/3} f \\
 1
 \end{pmatrix},\qquad
 s_3 = \begin{pmatrix}
 -2^{2/3}f \\
 1 \\
 0
 \end{pmatrix}.
\end{equation*}
Since the function $f$ has no zero in $\C^*$, $(s_1,s_2,s_3)$ is the frame of $E$. We see that
\begin{equation*}
 \big(\theta+\big(2^{-1/3}f {\rm d}z\big)/z\,\id_E\big)(s_1,s_2,s_3) = (s_1,s_2,s_3) \begin{pmatrix}
 3\cdot 2^{-1/3}f & 0 & 0\\
 0 & 0 & 1 \\
 0 & 0 & 0
 \end{pmatrix}\frac{{\rm d}z}{z}.
\end{equation*}
Thus, we can take the frame $(v_1,v_2,v_3)$ of $E$ such that
\begin{gather*}
 \big(\theta+\big(2^{-1/3}f {\rm d}z\big)/z\,\id_E\big)(v_1,v_2,v_3) = (v_1,v_2,v_3) \begin{pmatrix}
 3\cdot 2^{-1/3}f & 0 & 0\\
 0 & 0 & 1 \\
 0 & 0 & 0
 \end{pmatrix}\frac{{\rm d}z}{z},\\
 C(v_1,v_1)=1,\qquad C(v_i,v_i)=0\qquad i=2,3,\\
 C(v_2,v_3)=\begin{cases}
 1 & \text{if $f=az^2$},\\
 z & \text{if $f=az$}.
 \end{cases}
\end{gather*}
We define $\Tilde{E}$, $E_i$, $\Tilde{\theta}$ and $\Tilde{C}$ as in Section \ref{subsection_existence},
\begin{gather*}
 \tilde{E} = \bigoplus_{i=1}^3 \strsheaf_{\ps^1}(*D)v_i,\qquad E_1 = \strsheaf_{\ps^1}(*D)v_1,\qquad E_2 = \strsheaf_{\ps^1}(*D)v_2, \\
 E_3 = \strsheaf_{\ps^1}(*D)v_2\oplus \strsheaf_{\ps^1}(*D)v_3,\qquad E_4 = \strsheaf_{\ps^1}(*D)v_1\oplus \strsheaf_{\ps^1}(*D)v_2.
\end{gather*}
By Theorem \ref{correspond_thm}, it suffices to find a stable good filtered Higgs bundle over $\bigl(\Tilde{E},\Tilde{\theta}\bigr)$ such that~$\Tilde{C}$ is perfect. Under the change of coordinate $z\mapsto z_1=\alpha z$, ${\rm d}z/z$ remains unchanged, while~$\beta z^b$ transforms to $\beta \alpha^b z^b$. For this reason, we will consider filtered Higgs bundles over $\big(\Tilde{E},\varphi_b\big)$, where~$\varphi_b$ is the Higgs field satisfying
\begin{equation*}
 \varphi_b(v_1,v_2,v_3) = (v_1,v_2,v_3) \begin{pmatrix}
 z^b & 0 & 0\\
 0 & 0 & 1 \\
 0 & 0 & 0
 \end{pmatrix}\frac{{\rm d}z}{z}.
\end{equation*}
Let $\psi_b$ be the morphism defined as $\psi_b\,{\rm d}z/z=\varphi_b$.
\begin{Proposition}\label{prop_Cstar_b2}
 Let $u_1$, $u_2$ and $u_3$ be sections of $\Tilde{E}$ defined as follows:
 \begin{gather*}
 u_1 = z^{-2}v_1 + \frac{\sqrt{-1}}{2}z^{-3}v_2 + \sqrt{-1}z^{-1}v_3,\\
 u_2 = zv_2,\qquad u_3 = -\frac{1}{2} z^{-1}v_2 + zv_3.
 \end{gather*}
 We set a filtered Higgs bundle $\big(\mathcal{P}_{*}\Tilde{E},\varphi_2\big)$ as
 \begin{gather*}
 \mathcal{P}_c \Tilde{E}_{0} = \strsheaf_{\ps^1,0}([c]0)u_1\oplus \strsheaf_{\ps^1,0}([c]0)u_2\oplus \strsheaf_{\ps^1,0}([c]0)u_3,\\
 \mathcal{P}_c \Tilde{E}_{\infty} = \strsheaf_{\ps^1,\infty}([c]\infty)v_1\oplus \strsheaf_{\ps^1,\infty}([c]\infty)v_2\oplus \strsheaf_{\ps^1,\infty}([c]\infty)v_3.
 \end{gather*}
 Then $\big(\mathcal{P}_{*}\Tilde{E},\varphi_2\big)$ is a stable good filtered Higgs bundle with the perfect pairing $\Tilde{C}$.
\end{Proposition}
\begin{proof}
 We have
 \begin{alignat*}{4}
 &\psi_2(u_1)=z^2 u_1-\sqrt{-1}u_3,\qquad&& \psi_2(u_2)=0,\qquad&& \psi_2(u_3)=u_2,&\\
 &\psi_2(v_1)=z^2v_1,\qquad&& \psi_2(v_2)=0,\qquad&& \psi_2(v_3)=v_2.&
 \end{alignat*}
 These imply that $\bigl(\mathcal{P}_{*}\Tilde{E},\varphi_2\bigr)$ is good. Moreover, we have
 \begin{gather*}
 C(u_1,u_2)=\sqrt{-1},\qquad C(u_3,u_3)=-1,\qquad C(v_1,v_1)=1,\qquad C(v_2,v_3)=1.
 \end{gather*}
 Thus, the pairing $\Tilde{C}$ is perfect. We also have
 \begin{equation*}
 C(u_2,u_3)=z^2,\qquad v_1= z^2u_1-\sqrt{-1}z^{-2}u_2-\sqrt{-1}u_3.
 \end{equation*}
 These imply that the induced pairings $\Tilde{C}\colon \mathcal{P}_{*}E_1\otimes \mathcal{P}_{*}E_1\to \mathcal{P}^{(0)}_*(\mathcal{O}_{\ps^1}(*D))$ and $\Tilde{C}\colon \mathcal{P}_{*}E_3\otimes \mathcal{P}_{*}E_3 \to \mathcal{P}^{(0)}_*(\mathcal{O}_{\ps^1}(*D))$ are not perfect.
 Therefore, we obtain that $\deg(\mathcal{P}_{*}E_1)<0$ and $\deg(\mathcal{P}_{*}E_3)<0$. Since $u_2=zv_2$ is a local frame of $\mathcal{P}_{0,0}E_2$ on $\C$ and $v_2$ is a local frame of $\mathcal{P}_{0,0}E_2$ on $\ps^1\setminus \{0\}$, we~see~that
 \begin{equation*}
 \deg(\mathcal{P}_{*}E_2) = \deg(\mathcal{P}_{0,0}E_2) = -1.
 \end{equation*}
 The sections $z^2u_1-\sqrt{-1}u_3=v_1+\sqrt{-1}zv_2$ and $u_2=zv_2$ form a local frame of $\mathcal{P}_{0,0}E_4$ on $\C$. Thus, $zv_1\wedge v_2$ is a frame of $\det(\mathcal{P}_{0,0}E_4)$ on $\C$. The section $v_1\wedge v_2$ is a local frame of $\det(\mathcal{P}_{0,0}E_4)$ on $\ps^1\setminus \{0\}$. Therefore, we obtain that
 \begin{equation*}
 \deg(\mathcal{P}_{*}E_4) = \deg(\mathcal{P}_{0,0}E_4) = -1.
 \end{equation*}
 As a result, $\bigl(\mathcal{P}_{*}\Tilde{E},\varphi_2\bigr)$ is stable.
\end{proof}

\begin{Proposition}\label{prop_Cstar_b1}
 Let $u'_1$, $u'_2$ and $u'_3$ be sections of $\Tilde{E}$ defined as follows:
 \begin{gather*}
 u'_1 = z^{-1}v_1 + \frac{\sqrt{-1}}{2}z^{-2}v_2 + \sqrt{-1}z^{-1}v_3,\\
 u'_2 = v_2,\qquad u'_3 = -\frac{1}{2} z^{-1}v_2 + v_3.
 \end{gather*}
 We set a filtered Higgs bundle $\bigl(\mathcal{P}_{*}\Tilde{E},\varphi_1\bigr)$ as
 \begin{gather*}
 \mathcal{P}_c \Tilde{E}_{0} = \strsheaf_{\ps^1,0}([c]0)u'_1\oplus \strsheaf_{\ps^1,0}([c]0)u'_2\oplus \strsheaf_{\ps^1,0}([c]0)u'_3,\\
 \mathcal{P}_c \Tilde{E}_{\infty} = \strsheaf_{\ps^1,\infty}([c]\infty)v_1\oplus \strsheaf_{\ps^1,\infty}([c-1/2]\infty)v_2\oplus \strsheaf_{\ps^1,\infty}([c-1/2]\infty)v_3.
 \end{gather*}
 Then $\bigl(\mathcal{P}_{*}\Tilde{E},\varphi_1\bigr)$ is a stable good filtered Higgs bundle with the perfect pairing $\Tilde{C}$.
\end{Proposition}
\begin{proof}
 We have
 \begin{alignat*}{4}
 &\psi_1(u'_1) = zu'_1 - \sqrt{-1}u'_3,\qquad&& \psi_1(u'_2) = 0,\qquad&& \psi_1(u'_3) = u'_2,&\\
 &\psi_1(v_1) = zv_1, \qquad&& \psi_1(v_2) = 0, \qquad&& \psi_1(v_3) = v_2.&
 \end{alignat*}
 These imply that $\bigl(\mathcal{P}_{*}\Tilde{E},\varphi_1\bigr)$ is good. Moreover, we have
 \begin{gather*}
 C\big(u'_1,u'_2\big)=\sqrt{-1},\qquad C\big(u'_3,u'_3\big)=-1,\qquad C(v_1,v_1)=1,\qquad C(v_2,v_3)=z.
 \end{gather*}
 Thus, the pairing $\Tilde{C}$ is perfect. Since we also have
 \begin{equation*}
 C\big(u'_2,u'_3\big)= z, \qquad v_1 = zu'_1 -\sqrt{-1}z^{-1}u'_2 - \sqrt{-1}u'_3,
 \end{equation*}
 we obtain that $\deg(\mathcal{P}_{*}E_1)<0$ and $\deg(\mathcal{P}_{*}E_3)<0$ in the same way as Proposition \ref{prop_Cstar_b2}. Since~${u'_2=v_2}$ is a global frame of $\mathcal{P}_{0,1/2}E_2$, we see that
 \begin{equation*}
 \deg(\mathcal{P}_{*}E_2) = \deg(\mathcal{P}_{0,1/2}E_2) - \frac{1}{2} = -\frac{1}{2}.
 \end{equation*}
 The sections $zu'_1-\sqrt{-1}u'_3=v_1+\sqrt{-1}z^{-1}v_2$ and $u'_2=v_2$ form a global frame of $\mathcal{P}_{0,1/2}E_4$. Thus, $v_1\wedge v_2$ is a frame of $\det(\mathcal{P}_{0,1/2}E_4)$. Therefore, we obtain that
 \begin{equation*}
 \deg(\mathcal{P}_{*}E_4) = \deg(\mathcal{P}_{0,1/2}E_4) - \frac{1}{2} = -\frac{1}{2}.
 \end{equation*}
 As a result, $\bigl(\mathcal{P}_{*}\Tilde{E},\varphi_1\bigr)$ is stable.
\end{proof}

By Propositions \ref{prop_Cstar_b2}, \ref{prop_Cstar_b1}, Corollary \ref{main_Cstar_cor} and Proposition \ref{remark_degree}, we obtain Theorem \ref{main_Cstar'}.

\subsection*{Acknowledgement}
The author is grateful to the referees for their careful reading and valuable comments.
The author is grateful to his supervisor, Takuro Mochizuki, for many discussions and helpful advice.
The author thanks Qiongling Li and Takashi Ono for their helpful advice.
This work was supported by JST SPRING, Grant Number JPMJSP2110.

\pdfbookmark[1]{References}{ref}
\LastPageEnding

\end{document}